\documentclass[manyauthors]{fundam}

\usepackage{url} 
\usepackage[ruled,lined]{algorithm2e}
\usepackage{graphicx}
\usepackage{amsmath,amssymb,xy}
\usepackage{tikz}
\usetikzlibrary{automata, positioning, arrows}
\usepackage{mymacros}

   \usepackage{hyperref}

\begin{document}

\setcounter{page}{233}
\publyear{2021}
\papernumber{2098}
\volume{184}
\issue{3}

     \finalVersionForARXIV


\title{Theory of  Constructive Semigroups with Apartness -- \\ Foundations, Development and Practice}

\author{Melanija Mitrovi\'c\thanks{Address for correspondence: Faculty of Mechanical Engineering,
                    A. Medvedeva 14, 18000 Ni\v s, Serbia. \newline \newline
          \vspace*{-6mm}{\scriptsize{Received March 2021; \ accepted January 2022.}}}
 \\
        Faculty of Mechanical Engineering\\
           University of Ni\v s \\
        A. Medvedeva 14, 18000 Ni\v s, Serbia \\
        melanija.mitrovic@masfak.ni.ac.rs
\and    Mahouton Norbert Hounkonnou\thanksas{1}\\
        Int. Chair in Mathematical Physics and Applications\\
        (ICMPA-UNESCO Chair) \\
        University of Abomey-Calavi \\
        072 B.P. 50 Cotonou, Republic of Benin, Africa \\
        norbert.hounkonnou@cipma.uac.bj \\
        (with copy to hounkonnou@yahoo.fr)
\and    Marian Alexandru Baroni\thanksas{1}\\
        “Dunarea de Jos” University of Galati,
         Romania \\
         marian.baroni@ugal.ro
  }

\maketitle

\runninghead{M. Mitrovi\'c et al. }{Theory of Constructive Semigroups with Apartness  ...}

{\vspace*{3mm}\leftskip 1.5cm \scriptsize  ``Semigroups aren't a barren, sterile
flower on the tree of algebra, they are a natural algebraic approach
to some of the most fundamental concepts of algebra (and mathematics
in general), this is why they have been in existence for more then
half a century, and this is why they are here to stay.'' }
\\
\hbox{}\hfill \scriptsize {Boris M. Schein, \cite{bms}\vspace*{3mm}}

\begin{abstract}
 This paper has several purposes. We present through a critical
review the results from already published papers on the constructive
semigroup theory, and contribute to its further development by
giving solutions to open problems. We also draw attention to its
possible applications in other (constructive) mathematics
disciplines, in computer science, social sciences, economics, etc.
Another important goal of this paper is to provide  a clear,
understandable picture of constructive semigroups with apartness in
Bishop's style both to (classical) algebraists and the  ones who
apply algebraic knowledge.
\end{abstract}


\normalsize
\section{Introduction}

A general answer to the question what constructive mathematics is
could be formulated as follows: it is mathematics which can be
implemented on a computer. There are two main ways of developing
mathematics constructively. The first one uses classical traditional
logic  within  a strict algorithmic framework. The second way is to
replace classical logic with \emph{intuitionistic logic}.

Throughout this paper \emph{constructive mathematics} is understood
as mathematics performed in the context of intuitionistic logic,
that is, without the law of excluded middle (\textbf{LEM}). There
are two main characteristics for a constructivist trend. The notion
of \emph{truth} is not taken as primitive, and \emph{existence}
means constructibility. From the classical mathematics
(\textbf{CLASS}) point of view, mathematics consists of a
preexisting mathematical truth.  From a constructive viewpoint,  the
judgement $\varphi$ \emph{is true} means that \emph{there is a proof
of} $\varphi$.
 ``What constitutes  a proof is a social construct, an agreement
among people as to what is a valid argument. The rules of logic
codify a set of principles of reasoning that may be used in a valid
proof. Constructive (intuitionistic) logic codifies the principles
of mathematical reasoning as it is actually practiced,'' \cite{hr}.
In constructive mathematics, the \emph{status of existence
statement} is much stronger than in \textbf{CLASS}. The classical
interpretation is that an object exists if its non-existence is
contradictory. In constructive mathematics when the existence of an
object is proved, the proof also demonstrates how to find it. Thus,
following further \cite{hr}, the \emph{constructive logic} can be
described as logic of people matter, as distinct from the classical
logic, which may be described as the logic of the mind of God. One
of the main features of constructive mathematics is that the
concepts that are equivalent in the presence of \textbf{LEM}, need
not be equivalent any more. For example, we distinguish nonempty and
inhabited sets, several types of inequalities, two complements of a
given set, etc.

There is no  doubt about deep connections between constructive
mathematics and computer science. Moreover, ``if programming is
understood  not as the writing of instructions for this or that
computing machine but as the design of methods of computation that
is the computer's duty to execute, then it no longer seems possible
to distinguish the discipline of programming from constructive
mathematics'',  \cite{ml}.

Constructive mathematics is not a unique notion. Various forms of
constructivism have been developed over time. The principal trends
include the following varieties: \textbf{INT} - Brouwer's
intuitionistic mathematics, \textbf{RUSS} - the constructive
recursive mathematics of the Russian school of Markov, \textbf{BISH}
- Bishop's constructive mathematics. Every form  has intuitionistic
logic at its core. Different schools have different additional
principles or axioms given by the particular approach to
constructivism. For example, the notion of  an \emph{algorithm} or a
\emph{finite routine} is taken as primitive in \textbf{INT} and
\textbf{BISH}, while \textbf{RUSS} operates with a fixed programming
language and an algorithm is a sequence of symbols in that language.
We have to emphasize that  Errett Bishop - style constructive
mathematics, \textbf{BISH},  forms the framework for our work.
\textbf{BISH}  enables one to interpret the results both in
classical mathematics and in other varieties of constructivism.
\textbf{BISH} originated in 1967 with the publication of the book
\emph{Foundations of Constructive Mathematics},  \cite{eb}, and with
its second, much revised edition in 1985, \cite{bdb}. There has been
a steady stream of publications contributing to Bishop's programme
since 1967.  A ten-year long systematic research of computable
topology, using apartness as the fundamental notion, resulted in the
first book, \cite{dbv2}, on topology within \textbf{BISH} framework.
\texttt{Modern algebra}, as is noticed
\eject
\noindent  in \cite{dbre}, ``contrary to
Bishop's expectations, also proved amenable to natural,
thoroughgoing, constructive treatment''.

 Working within the classical theory of semigroups
 several years ago we, \cite{mm}, decided to change the  classical
background with the intuitionistic one. This meant, among other
things, that the perfect safety of the  classical theory with
developed notions, notations and methodologies was left behind.
Instead, we embarked on an adventure into exploring an algebraically
new area (even without clearly stated notions and notations) of
\emph{constructive semigroups with apartness}. What we had ``in
hand'' at that moment was the experience and knowledge coming from
the classical semigroup theory, other constructive mathematics
disciplines  such as, for example, constructive analysis, and,
especially, from constructive topology, as well as constructive
theories of groups and rings with tight apartness and computer
science. For classical algebraists who, like us, wonder ``on the odd
day''  what constructive algebra is all about, and  who want to find
out what it  feels like doing it, they will understand soon that
\emph{constructive algebra} is more complicated than classical
algebra in various ways: algebraic structures as a rule do not carry
a decidable equality relation (this difficulty is partly met by the
introduction of a strong inequality relation, the so-called
apartness relation); there is (sometime) the awkward abundance of
all kinds of substructures, and hence of quotient structures,
\cite{t1}.

As highlighted by Romano DA \emph{et al},
\cite{mmscr},
 \emph{the theory of semigroup with apartness} is a new
approach to semigroup theory and not a new class of semigroups. It
presents a semigroup facet of some relatively well established
direction of constructive mathematics which, to the best of our
knowledge, has not yet been considered within the semigroup
community. This paper has several purposes: to present through a
critical review results from our already published papers,
\cite{cmr1}, \cite{cmr2}, \cite{mmscr}, on the constructive point of
view on semigroup theory, to contribute to its further development
giving solutions to the problems posted (in the open, or, somehow
hidden way) within the scope of those papers, and to lay the
foundation for further works.

The order theory provides one of the most basic tools of semigroup
theory within \textbf{CLASS}. In particular, the structure of
semigroups is usually most clearly revealed through the analysis of
the behaviour of  their appropriate orders.  The most basic concept
leads to the \emph{quasiorders}, reflexive and transitive relations
with the fundamental concepts being introduced whenever possible in
their natural properties. Going through \cite{cmr1}, \cite{cmr2},
\cite{mmscr}, we can conclude that one of the main objectives of
those papers is to develop an appropriate constructive order theory
for semigroups with apartness. We outline some of the basic concepts
of semigroups with apartness, as special subsets on the one hand,
and orders on the other. The strongly irreflexive and
 co-transitive relations are building blocks of the constructive order
 theory we develop. With a primitive notion of 'set with apartness'
 our main intention was to connect all relations defined on such a
 set.  This is done by requiring them to be a part (subset) of an apartness.
 Such a relation is clearly strongly irreflexive. If, in addition,
 it is  co-transitive, then it is called \emph{co-quasiorder}.

In  algebra within \textbf{CLASS},  the formulation of homomorphic
images (together with substructures and direct products) is one of
the principal tools used to manipulate algebraic structures. In the
study of homomorphic images of an algebraic structure, a lot of help
comes from the notion of a\textbf{\ }quotient structure, which
captures all homomorphic images, at least up to isomorphism. On the
other hand, the homomorphism is the concept which goes hand in hand
with congruences. The relationship between quotients, homomorphisms
and congruences is described by the celebrated \emph{isomorphism
theorems}, which are a general and important foundational part of
abstract and universal algebras. The quotient structures are not
part of \textbf{BISH}. The quotient structure does not, in general,
have a natural apartness relation. So, \emph{the Quotient Structure
Problem} (\textbf{QSP}) is one of the very first problem which has
to be considered for any structure with apartness. The solutions of
\textbf{QSP} problem for sets and semigroups with apartness were
given in \cite{cmr1}. Some examples of special cases
can be found in \cite{dar-05} and \cite{dar-07}.
\emph{Co-equivalences}, symmetric co-quasiorders, and equivalences
which can be associated to them play the main roles. As an example
that a single concept of classical mathematics may split into two or
more distinct concepts when working constructively we have logical
$\neg Y$ and apartness $\sim Y$ complement of a given subset $Y$ of
a set or semigroup with apartness. The key for the solution of the
\textbf{QSP} for a set and semigroup with apartness is given by the
next theorem (Theorem 2.3, \cite{cmr1}).

\begin{teo}\label{cmr1th1}
If $\kappa$ is a co-equivalence on $S$, then the relation
${\sim}\kappa$ is an equivalence on $S$, and $\kappa$ defines
apartness on  $S/\,\sim\kappa$.
\end{teo}

Theorem~\ref{cmr1th1} is the key ingredient to formulate and prove
the apartness isomorphism theorem for a set with apartness (see
\cite{cmr1},  Theorem 2.5). Based on these results, the apartness
isomorphism theorem for a semigroup with apartness (\cite{cmr1},
Theorem 3.4) is formulated and proved as well. The just mentioned
results are significantly improved in
 \cite{mmscr}, where, among other things, it is proved that the two
 complements, logical and apartness, coincide for a  co-quasiorder $\tau$
 defined on a set or semigroup with apartness, i.e. we have $\sim\tau = \neg\tau$
 (see Proposition 2.3, Theorem 2.4 from \cite{mmscr}).

 \begin{rem}
 In \cite{mmss2}, an overview to the development of isomorphism
 theorems in certain algebraic structures - from classical to
 constructive - is given.
 \end{rem}

It is well known that within \textbf{CLASS } a number of subsets of
a semigroup enjoy special properties relative to  multiplication,
for example, completely isolated subset, convex subset,
subsemigroup, ideal. On the other hand, relations defined on a
semigroup can be distinguished one from another according to the
behaviour of their related elements  to multiplication. From that
point of view,  \emph{positive quasiorders} are  of special
interest. Going through literature with the theory of semigroups as
the main topic, one can see that there is almost no method of
studying semigroups without a certain type of positive quasiorders
involved. It is often the case that results on connections between
positive quasiorders and subsets defined above showed them fruitful
as well. Partly inspired by classical results, in \cite{cmr2} we
consider \emph{complement positive co-quasiorders}, i.e.
constructive counterparts of positive quasiorders, and their
connections with special subsets of semigroups with apartness.
Inspired by the existing notion from constructive analysis and
topology, we use the complements (both of them) for the
classification of subsets of a given set with apartness.
\emph{Strongly detachable subsets}, i.e. those subsets for which we
can decide whether an element $x$ from that set belongs to the
subset in question or to its apartness complement, play a
significant role within the scope of \cite{cmr2}. In Lemma 3.2, we
prove that for any co-quasiorder $\tau$, defined on a set with
apartness, its left and right $\tau$-classes of any element are
strongly detachable subsets. The main result of this paper, Theorem
4.1, gives the description of a complement positive co-quasiorder,
defined on a semigroup with apartness, via the behaviour of its left
and right classes and their connections with special subsets.

Apart from the above issues, an addition problem is the so-called
\emph{constant domain axiom}: the folklore
 type of axiom in \textbf{CLASS} algebra
$$\varphi\,\vee\, \forall  x \, \psi (x) \ \leftrightarrow \ \forall
x\, (\varphi\,\vee\,  \psi (x)),$$ its constructive version
$$\varphi\,\vee\, \forall _x \, \psi (x) \ \rightarrow \ \forall
_x\, (\varphi\,\vee\,  \psi (x))$$ can be a source of problems when
doing algebra constructively. We choose intuitionistic logic of
constant domains \textbf{CD }to be the background of \cite{cmr2}.
Recall that  intermediate logic, such as, for example \textbf{CD},
the logic that is stronger than  intuitionistic logic but weaker
than  classical one, can be constructed by adding one or more axioms
to intuitionistic logic. There is a continuum of such logics.
For more details see \cite{fa}, \cite{ag}.

\medskip
The presence of apartness implies the appearance of different types
of substructures connected to it. We deal with strongly detachable
subsets  in \cite{cmr2}. In \cite{mmscr} we mentioned two more:
detachable and quasi-detachable subsets. In Proposition 2.1 we show
that a strongly detachable subset is detachable and
quasi-detachable. Even more, from \cite{mmscr} the apartness and
logical complements coincide for strongly detachable and
quasi-detachable subsets.

\medskip
Going through \cite{cmr2} and \cite{mmscr} it can be noticed that we
can face  several problems arising from their scope. For example,
\begin{itemize}
\item The relations between detachable, strongly detachable and quasi
detachable subsets are only partially described in \cite{mmscr},
Proposition 2.1. A complete description of their relationship
remains an open problem.

\item Are  the results of \cite{cmr2} valid in  intuitionistic logic if
we work with quasi-detachable subsets instead of strongly detachable
ones? Which of the presented results or their form(s) are valid for
the intuitionistic background, if any?
\end{itemize}

To conclude, the theory of semigroup with apartness,  its background
and motivations, further development and its  possible applications
as well as the critical answers to a number of questions including
those mentioned above  will be the main topics  throughout this
paper.

The paper is organized in the following way.  In our work on
constructive semigroups with apartness, as it is pointed out above,
we have faced an algebraically completely new area. The background
and motivation coming from the classical semigroup theory, other
constructive mathematics disciplines and computer science are the
content of   \textbf{Section 2}. Some results on classical
semigroups which can partly be seen as an inspiration for the
constructive ones are also discussed here. In \textbf{Section 3},
the main one, we are going to give a critical review of some of the
published results on sets and semigroups with apartness as well as
the solutions to some of the open problems  on sets and semigroups
with apartness. One of the main results, Theorem~\ref{senegcom},
gives a complete description of the relationships between
distinguished subsets of a set with apartness, which, in turn,
justifies the constructive order theory we develop with those
subsets with the main role in that framework. By
Proposition~\ref{p11-bhm}, if  any left/right-class of a
co-quasiorder defined on a set with apartness  is  a (strongly)
detachable subset then the limited principle of omniscience,
\textbf{LPO}, holds. This shows that Theorem 4.1 (and Lemma 3.2
important for its proof) set in \cite{cmr2}
 cannot be proved in \textbf{BISH} without the logic of constant domains \textbf{CD}.
Within intuitionistic logic, we can prove its weaker version,
Theorem~\ref{thm4.11}, which is another important result of this
section. As for \textbf{QSP}, for sets and semigroups with
apartness, we achieve a little progress in that direction.
Theorem~\ref{BHM-1}, the key theorem for the \textbf{QSP}'s
solution, generalizes the similar ones from \cite{cmr1},
\cite{mmscr}. In addition, as a generalization of the first
apartness isomorphism theorem, the new theorem,
Theorem~\ref{basicfactor}, the second apartness isomorphism theorem
for sets with apartness, is formulated and proved. Finally, in
\textbf{Section 4}, examples of some already existing applications
as well as  future possible realizations of the ideas presented in
the previous section are given.

More background on constructive mathematics can be found
in \cite{mjb}, \cite{eb}, \cite{dbv2}, \cite{t1}. The standard
reference for constructive algebra is \cite{mrr}. For the classical
case see  \cite{mm}, \cite{p}. Examples of applications of these
theoretical concepts can be found in \cite{bb}, \cite{bh},
\cite{gc}, \cite{gpwz2},  \cite{mam}.

\section{Preliminaries: background, known results  and motivation}\label{s-2}

Starting our work on constructive semigroups with apartness, as
 pointed out above,  we  faced   an algebraically completely new
area. What we had in ``hand'' at that moment were the experience and
knowledge coming from the classical semigroup theory, other
constructive mathematics disciplines, and  computer science.

\subsection{Algebra and semigroups within \textbf{CLASS}}\label{s-21}

{\leftskip 1cm \scriptsize ``I was just going to say, when I was
interrupted, that one of the many ways of classifying minds is under
the heads of arithmetical and algebraical intellects. All economical
and practical wisdom is an extension of the following arithmetical
formula: 2 + 2 = 4. Every philosophical proposition has the more
general character of the expression a + b = c. We are mere
operatives, empirics, and egotists until we learn to think in
letters instead of figures .'' \vspace*{-0.3cm}}
\begin{flushright}
\scriptsize Oliver Wendell Holmes: \emph{The Autocrat of the
Breakfast Table} \end{flushright}

A very short account of   abstract algebra and its development will
be given here. Over the course of the 19th century, algebra made a
transition from a subject concerned entirely with the solution of
mostly polynomial equations to a discipline that deals with general
structures within mathematics.  The term abstract algebra as a name
 for this area appeared in the early 20th century.  ``In studying abstract
algebra, a so called  axiomatic approach is taken; that is, we take
a collection of objects $S$ and assume some rules about their
structure. These rules are called axioms. Using the axioms for $S$,
we wish to derive other information about $S$ by using logical
arguments. We require that our axioms be consistent; that is, they
should not contradict one another. We also demand that there not be
to many axioms. If a system of axioms is too restrictive, there will
be few examples of the algebraic structure,''~\cite{jtw}.

 An \textbf{\emph{algebraic structure}} can be, informally,
described as a set of some elements of objects with some (not
necessarily, but often, binary) operations for combining them. A
\emph{set} is considered as a primitive notion which one does not
define. We will take the intuitive approach that a set is some given
collection of objects, called elements or members of the set. The
cartesian product of a set $S$ with itself, $S\times S$, is of
special importance. A subset $\rho$ of $S\times S$, or,
equivalently, a property applicable to elements of $S\times S$, is
called a \emph{binary relation on} $S$. The ordered pair $(S, \rho)$
is a particular \emph{relational structure}. In general,  there are
many properties (for example: reflexivity, symmetry, transitivity)
that binary relations may satisfy on a given set. As usual, for a
relation $\rho$ on $S$, $a\rho = \{x\in S : (a,x)\in\rho \}$, and
$\rho  a   = \{x\in S : (x,a)\in\rho \}$ are the left and the right
$\rho$-class of the element $a\in S$ respectively. The concept of an
\emph{equivalence}, i.e. reflexive, symmetric and transitive
relation, is an extremely important one and plays a central role in
mathematics. If $\varepsilon$ is an equivalence on a set $S$, then
$S/\varepsilon = \{x\varepsilon : x\in S\}$ is called the
\emph{quotient set of} $S$ \emph{by} $\varepsilon$. Classifying
objects according to some property is a frequent procedure in many
fields. Grouping elements in ``a company'' so that elements  in each
group are of the same prescribed property as performed by
equivalence relations, and the classification gives the
corresponding quotient sets. Thus, abstract algebra can show us how
to identify objects with the same properties properly - we have to
switch to a quotient structure (technique applicable, for example,
to abstract data type theory).

Some fundamental  concepts in abstract algebra are: set and
operation(s) defined on that set; certain algebraic laws that all
elements of  a structure can respect, such as, for example,
associativity, commutativity;  some elements with special behaviour
in connection with operation(s): idempotent elements, identity
element, inverse elements, ... Combining the above concepts gives
some  of the most important structures in mathematics:  groups,
rings, semigroups, ... Centred around an algebraic structure are
notions of: substructure, homomorphism, isomorphism, congruence,
quotient structure. A mapping  between  two algebraic structures of
the same type, that preserves the operation(s) or  is compatible
with the operation(s) of the structures is called
\emph{homomorphism}. Homomorphisms are essential to the study of any
class of algebraic objects. An equivalence relation $\rho$ on an
algebraic structure $S$ (such as a group, a ring, or a semigroup)
that is compatible with the structure is called a \emph{congruence}.
Within \textbf{CLASS} the quotient set $S /\rho$ becomes the
structure of the same type in a natural way. The relationship
between quotients, homomorphisms and congruences is described by the
celebrated \textbf{\emph{isomorphism theorems}}. Isomorphism
theorems are a general and important foundational part of abstract
and universal algebra.

``Algebra is beautiful. It is so beautiful that many people forget
that algebra can be very useful as well,'' \cite{lp}. Abstract
algebra is the highest level of abstraction. Understanding it means,
among other things, that one can think more clearly, more
efficiently. With the development of computing in the last several
decades, applications that involve  algebraic structures have become
increasingly important. To mention a few, lot of data structures
form monoids (semigroups with the identity element); algebraic
properties are  important for parallel execution of programs - for
example, combining a list of items with some binary operators can be
easily parallelized if that operator is associative (commutativity
is often required as well). Examples of applications given above
lead to \emph{semigroups}. In fact, following \cite{gl}, (free)
\emph{semigroups} are the first mathematical objects every human
being has to deal with - even before attending  school.

\subsubsection {More about the theory of semigroups}\label{s-211}

A \textbf{\emph{semigroup}} is an algebraic structure consisting of
a set with an associative binary operation defined on it.  In the
history of mathematics, \emph{the algebraic theory of semigroups} is
a relative newcomer, with the theory proper developing only in the
second half of the twentieth century. Historically, it can be viewed
as an algebraic abstraction of the properties of the composition of
transformations on a set.   But, there is no doubt about it, the
main sources came from group  and ring theories. However, semigroups
are not a direct generalization of group theory as well as ring
theory. Let us remember: congruences on groups are uniquely
determined by their normal subgroups, and, on the other hand, there
is a bijection between congruences and the ideals of  rings. The
study of congruences on semigroups is more complicated - no such
device is available. One must study congruences as such. Thus,
semigroups do not much resemble groups and rings. In fact,
semigroups do not much resemble any other algebraic structure.
Nowadays, semigroup theory is an enormously broad topic and has
advanced on a very broad front.
 Following \cite{amgp}, ``a huge variety of structures
studied by mathematicians are sets endowed with associative binary
operation.''  Even more, it appears that ``semigroup theory provides
a convenient general framework for unifying and clarifying a number
of topics in fields that are seen, at first sight, unrelated'',
\cite{gl}.

The capability and flexibility of semigroups from the point of view
of modeling and   problem-solving in extremely diverse situations
have been already pointed out, and  interesting new algebraic ideas
arise with binary applications  and connections to other areas of
mathematics and sciences. Let us start our short journey through the
applications of semigroups with the connections to the algebra of
relations. The theory of  semigroups is one of the main algebraic
tools used in the theory of automata as well as the theory of formal
languages.  According to some authors, the role of the theory of
semigroups for theoretical computer science is compared with the one
which the philosophy  has with the respect to science in general.
Some investigations on transformation semigroups of synchronizing
automata show up interesting implications for various applications
for robotics, or more precisely, robotic manipulation.  On the other
hand, areas such as biology, biochemistry, sociology also make use of
 semigroups. For example, semigroups can be used in biology to describe certain aspects
in the crossing of organisms, in genetics, and in consideration of
metabolisms. Following \cite{boo}, \cite{lp}, the sociology includes
the study of human interactive behaviour in group situations, in
particular in underlying structures of societies. The study of such
relations can be elegantly formulated in the language of semigroups.
The book \cite{bjp} is written for social scientists with the main
aim to help readers to apply ``interesting and powerful concepts''
of semigroup theory to their own fields of expertise. However, the
list of applications given above  does not purport to mention all of
the existing applications of semigroup theory.
  As it is pointed out in \cite{amgp}, it is often the case  that
  ``most applications make minimal use of the reach of the (classical)
  algebraic theory of semigroups.'' There is need for
study of some more structures of semigroups which can find
applications in different areas, \cite{rd}. This can bring
 very pretty mathematics to illustrate the
interplay between certain scientific areas and semigroup-theoretic
techniques. This type of research can be a topic on its own for
certain types of papers.


In what follows some known results from the classical semigroup
theory useful for our development   will be presented.

\medskip
\emph{A semigroup} $(S,  \cdot )$ is a set $S$ together with an
associative binary operation $\cdot$\begin{itemize}
\item[(A)] $ \enspace\quad (\forall  a,b,c \in S) \ [(a\cdot b)\cdot c\, =\, a\cdot (b\cdot
c)]$.
\end{itemize}
\noindent Where the nature of the
 multiplications is clear from the context, it is written $S$ rather
than $(S, \, \cdot )$.  Frequently, $xy$ is written rather than
$x\,\cdot\,y$.

Various approaches  have been developed over the years to construct
frameworks for understanding the structure of semigroups. The
fundamental concepts of  semigroup theory elaborated by
Suschekewitsch, Rees, Green, Clifford and other pioneers include as
one of the main tools, \emph{Green's quasiorders} (and equivalences
generated by them), defined by the multiplication of semigroups  and
in terms of special subsemigroups. The notion of an order plays an
important role throughout mathematics as well as in some adjacent
disciplines such as logic and computer science. Order theory
provides one of the most basic tools of semigroup theory as well. In
particular, the structure of semigroups is usually most clearly
revealed through the analysis of the behaviour of  their appropriate
orders. A pure order theory is concerned with a single undefined
binary relation $\rho$. This relation is assumed to have certain
properties (such as, for example, reflexivity, transitivity,
symmetry, antisymmetry), the most basic of which leads to the
concept of \emph{quasiorder}. A quasiorder plays a central role
throughout this short exposition with the fundamental concepts being
introduced whenever possible in their natural properties.

\medskip
\textbf{Distinguishing subsets}

\smallskip
A number of subsets of a semigroup enjoy special properties relative
to the multiplication. A subset $T$ of a semigroup $S$ is:

\begin{itemize}
  \item \emph{completely isolated} if $ab\in T$  implies $a\in T$ or $b\in T$ for any $a,b\in S$,
    \item \emph{convex} if $ab\in T$ implies both $a, b\in T$  for any $a,b\in S$,
   \item \emph{subsemigroup} if for any $a,b\in T$ we have  $ab\in T$,
    \item \emph{ideal} if for any $a\in T$ and $s\in S$ we have  $as,sa\in
    T$.
\end{itemize}
A  subsemigroup  $T$ of $S$ which is convex (resp. completely
isolated) as a subset is called a \emph{convex} (resp.
\emph{completely isolated}) \emph{subsemigroup}. In an analogous
way, we define a complex (completely isolated) ideal  of $S$. Some
of their existing properties are listed in the lemma below.

\begin{lem}\label{ccic}
Let $S$ be a semigroup. Then:

\begin{itemize}
\item[\emph{(i)}] An ideal $I$ of $S$ is completely isolated if and
only if $\neg I = S\setminus I$ is either a subsemigroup of $S$ or
is empty.
\item[\emph{(ii)}] An nonempty subset $F$ of $S$ is convex if and
only if $\neg F = S\setminus F$ is either a completely isolated
ideal or is empty.
\end{itemize}
\end{lem}

Within \textbf{CLASS } semigroups can historically  be viewed as an
algebraic abstraction of the transformations on a set. Of great
importance is the role of the subsemigroups given above  in
describing the structure of transformation semigroups. We refer the
reader  to \cite{gm}, \cite{ls} for more details about definitions,
properties and applications of such  subsemigroups.

\medskip
Describing a semigroup and its structure is a formidable task. There
are many different techniques developed for that purpose.
\emph{Semilattice decomposition of semigroups} is one of the methods
with general applications. For more information on semilattice
decomposition of semigroups see  \cite{mm},  \cite{p}. It is shown
in \cite{p} that this method leads to the study of completely
isolated  ideals and convex subsemigroups.

\eject
\textbf{ Quasiorders }

\smallskip
By   definition, a binary relation $\rho $ of set $S$ is a subset of
$S\times S$. To describe the relation defined on a semigroup $S$, we
have to say which order pairs belong to $\rho$. In other words, for
any $a\in S$, we have to know the   following subsets of $S$:

$a\rho \,  =  \, \{x\in S \, \mid \, (a,x)\in\rho \}$,

$\rho  a \,  =  \, \{x\in S \, \mid \, (x,a)\in\rho \}$,

\noindent called the left and right $\rho$-class of an element $a$.
That is how we connect a study of binary relations defined on a
given set with its subsets.

\medskip
The relations defined on a semigroup $S$ are distinguished one from
another according to the behaviour of their related elements  to the
multiplication. A relation $\rho$ defined on a semigroup $S$ is
\begin{itemize}
\item \emph{positive} if $(a,ab), (a,ba)\in\rho$, for any $a,b\in
S$,
\item \emph{with common multiply property}, or, also called for short, with
\emph{cm-property} if $(a,c), (b,c)\in \rho$ implies $(ab,c)\in
\rho$, for any $a,b,c\in S$.
\item \emph{with compatibility property} if
$(x,y), (u,v)\in\rho$ implies $(xu,yv)\in\rho$ for any $x,y,u,v\in
S$.
\end{itemize}

In the sequel, the positive quasiorders will also be considered.
Recall that the \emph{division relation} $|$ on a semigroup $S$,
defined by
$$a\,|\,b \ \stackrel{\rm def}{\Leftrightarrow} \ (\exists x,y\in S^1) \ b=xay,$$
for $a,b\in S$, is the smallest positive quasi-order defined on $S$.
The positive quasi-orders were introduced in \cite{s}. In \cite{t}
their link to semilattice decompositions of semigroups was
established. In \cite{msp1} their possible applications in
psychology were announced. Finally, close connections between
positive quasiorders and subsemigroups defined above were given in
\cite{bc}.  Here we mention some of these results:

\begin{teo}\label{bsc1}
Let $\rho$ be a quasiorder on $S$. The following conditions on a
semigroup $S$ are equivalent:
\begin{itemize}
\item[\emph{(i)}]  $\rho$ is a positive quasiorder;
\item[\emph{(ii)}] $(\forall a,b\in S) \, (ab)\rho\subseteq a\rho \cap
b\rho$;
\item[\emph{(iii)}] $(\forall a,b\in S) \, \rho a\cup \rho b\subseteq \rho (ab)$;
\item[\emph{(iv)}] $a\rho$  is an ideal for any $a\in S$;
\item[\emph{(v)}]  $\rho a$ is a convex subset of $S$ for any $a\in
S$.
\end{itemize}
\end{teo}

\begin{teo}\label{bsc2}
Let $\rho$ be a quasiorder on $S$. The following conditions on a
semigroup $S$ are equivalent:
\begin{itemize}
\item[\emph{(i)}]  $\rho$ is a positive quasiorder with cm-property;
\item[\emph{(ii)}] $\rho a$ is a convex subsemigroup of $S$ for any
$a\in S$;
\item[\emph{(iii)}]  $(\forall a,b\in S) \,  (ab)\rho = a\rho \cap
b\rho$.
\end{itemize}
\end{teo}

Finally, we can say that, from the point of view of the classical
semigroup theory, the interrelations between the following notions
are of interest:
\begin{itemize}
\itemsep=0.9pt
  \item semilattice decomposition of semigroups,
  \item completely isolated and convex subsemigroups and/or ideals,
  \item positive quasiorders.
\end{itemize}

\medskip
\textbf{Isomorphism theorems for semigroups}
\smallskip

Let us remember that congruences on groups are uniquely determined
by  their normal subgroups, and, on the other hand, there is a
bijection between congruences and the ideals of  rings. The study of
congruences on semigroups is more complicated - no such device is
available. One must study congruences as such. A \emph{congruence}
$\rho$ on a semigroup $S$ is an equivalence, i.e. symmetric
quasiorder,  with the \emph{compatibility property}. Classically,
the quotient set $S/\rho$ is then provided with a semigroup
structure.

\begin{teo}\label{sth21}
Let $S$ be a semigroup and $\rho$ a congruence on it.Then $S/\rho$
is a  semigroup with respect to the operation defined by
 $(x\rho)(y\rho) = (xy)\rho$, and the  mapping $\pi\,:\, S\,\rightarrow\,
S/\rho$, $\pi(x)=x\rho$, $x\in S$, is an \emph{onto} homomorphism.
\end{teo}

Provided that the (\emph{first}) \emph{isomorphism theorem for
semigroups} follows.

\begin{teo}\label{sth2}
Let $f: S \rightarrow T$ be a homomorphism between semigroups $S$
and $T$.  Then
\begin{itemize}
\itemsep=0.95pt
\item[\emph{(i)}] $ker\,f = f\circ f^{-1}=  \{(x,y)\in S\times S\,:\,
f(x)=f(y)\}$ is a congruence on $S$;
\item[\emph{(ii)}] the mapping $\theta \,:\, S/ker\,f \, \rightarrow
\, T$ defined by $\theta(x(ker\,f)) = f(x)$ is  an embedding such
that $f=\theta\circ\pi$;
\item[\emph{(iii)}]  if $f$ maps $S$ onto $T$, then $\theta$ is
an isomorphism.
\end{itemize}
\end{teo}

The theorem which follows is concerned with a more general
situation.

\begin{teo}\label{sth4}
Let $\rho$ be a congruence on a semigroup  $S$, and let  $f: S
\rightarrow T$ be a homomorphism between semigroups $S$ and $T$ such
that $\rho\subseteq ker\,f$. Then there exists a homomorphism of
semigroups $\theta : S/\rho  \rightarrow T$, such that
$f=\theta\circ\pi$. If, in addition,
 $f$  is  \emph{onto}, then $\theta$ is an isomorphism.
\end{teo}

\subsection{Constructive algebra}\label{s-22}

Constructive algebra is a relatively old discipline developed among
others by L. Kronecker, van der Waerden, A. Heyting. For more
information on the history see \cite{mrr}, \cite{t1}.
 One of the main topics in constructive algebra is constructive
algebraic structures with the relation of (tight) apartness $\#$,
the second most important relation in constructive mathematics. The
principal novelty in treating basic algebraic structures
constructively is that (tight) apartness becomes a fundamental
notion. (Consider the reals: we cannot assert that $x^{-1}$ exists
unless we know that $x$ is apart from zero, i.e. $|x| >0$ -
constructively that is not the same thing as $x\neq 0$. Furthermore,
in fields $x^{-1}$ exists only if $x$ is apart from 0, \cite{mjb})
 The study of algebraic structures in the presence of \emph{tight}
apartness  was started by  Heyting, \cite{ah3}. Heyting gave the
theory a firm base in \cite{ah4}. Roughly, the descriptive
definition of a structure with apartness includes two main parts:

\begin{itemize}
\item the notion of a certain classical algebraic structure is
straightforwardly adopted;

\item a structure is equipped with an apartness with standard operations
respecting that apartness.
\end{itemize}

Quotient structures are not part of \textbf{BISH}. A quotient
structure does not, in general, have a natural apartness relation.
So, \emph{the Quotient Structure Problem} - \textbf{QSP}  is one of
the very first problems which has to be considered for any structure
with apartness. Talking about the \textbf{QSP }for
sets and semigroups with apartness and its history - solution of the
\textbf{QSP} for sets with apartness is for the first time given in
\cite{cmr1}. The \textbf{QSP}'s solutions  for groups with
\emph{tight} apartness and commutative rings with \emph{tight}
apartness presented in \cite{mjb}, \cite{ddv}, \cite{fr2},
\cite{dar-88}, \cite{dar-88-1}, \cite{wr1}, \cite{t1} inspired us to
give solutions of \textbf{QSP} for \emph{sets  with apartness} in
2013, \cite{cmr1}, which, in turn, imply the solution of QSP for
semigroups with apartness as its consequence.

  A lot of ideas, notions and notations come from, for example,
the constructive analysis, and, especially, from the constructive
topology, as well as from constructive theories of groups and rings
with \emph{tight} apartness. Although  the area of constuctive
semigroups with apartness is still in its  infancy, we can already
conclude that, similarly to the clasical case, the semigroups with
apartness do not much resemble groups and rings. In fact, they do
not much resemble any other constructive  algebraic structures with
apartness.

\subsection{Computer science}\label{s-23}

  It is well known that formalization is a general method in
science. Although it was created as a technique in logic and
mathematics, it has entered into engineering as well. Formal
engineering methods can be understood as mathematically-based
techniques for the functional specification, development and
verification in the engineering of software and hardware systems.
Despite some initial suspicion, it  was proved that formal methods
are powerful enough to deal with real life systems. For example, it
is shown that ``software of the size and complexity as we find in
modern cars today can be formally specified and verified by applying
computer based tools for modeling and interactive theorem proving,''
\cite{mb}.

Proof assistants are computer systems which give a user the
possibility to do mathematics on a computer: from (numerical and
symbolical) computing aspects to the aspects of defining and
proving.  The latter ones, doing proofs, are  the main focus. It is
believed that, besides their great future within the area of
mathematics formalization, their applications within computer-aided
modelling and verification of the systems are and will be more
important.  One of the most popular, with the intuitionistic
background, is the proof assistant computer system \emph{Coq}.

Coq is used  for formal proves of   well known mathematical
theorems, such as, for example,  the Fundamental Theorem of Algebra,
FTA, \cite{gwz}. For that purpose, the  \emph{constructive algebraic
hierarchy} for Coq was developed, \cite{gpwz2}, consisting of
constructive basic algebraic structures (semigroups, monoids,
groups, rings, fields) with \emph{tight} apartness. In addition, all
these structures are limited to the commutative case. As it is
noticed in \cite{gpwz2} ``that algebraic hierarchy has been designed
to prove FTA. This means that it is not rich  as one would like. For
instance, we do not have noncommutative structure because they did
not occur in our work.''  ...  So, a question which arises from this
is:
\begin{quote}
\emph{What can be done in connection with noncommutative semigroups
with apartness where apartness
 is  only ``ordinary'' and  not the tight one?}
\end{quote}
We put noncommutative constructive semigroups with ``ordinary''
apartness in the core of our study, proving first, of course, that
such  semigroups do exist, \cite{cmr1}. As in \cite{bdb}, we made
``every effort to follow classical development along the lines
suggested by familiar classical theories or in all together new
directions.''

The results of our several years long investigations, \cite{cmr1},
\cite{cmr2}, \cite{mmscr},  present a semigroup facet of some
relatively well established directions of constructive mathematics
which, to the best of our knowledge, have not yet been considered
within the semigroup community. The initial step towards grounding
the theory done through our papers will be developed through the
scope of this paper. We are going to give a critical review of some
of those results as well as the solutions to some of the open
problems arising from those papers.

\section{Main results: sets and semigoups with apartness}\label{s-3}

Before starting our constructive examination of sets and semigroups
with apartness, we should clarify its setting. By constructive
mathematics we mean Bishop-style mathematics, \textbf{BISH}. We
adopt Fred Richman's viewpoint, \cite{fr1}, where  constructive
mathematics is simply mathematics carried out with intuitionistic
logic. The Bishop-style of constructive mathematics enables one to
interpret the results both in classical mathematics, \textbf{CLASS},
and other varieties of constructivism. We regard classical
mathematics as Bishop-style mathematics plus the law of excluded
middle, \textbf{LEM}.  This logical principle   can be regarded as
the main source of nonconstructivity. It was Brouwer, \cite{lejb},
 who first observed that  \textbf{LEM} was extended without
justification to statements about infinite sets. Several
consequences of \textbf{LEM} are not accepted in Bishop's
constructivism. We will mention three such nonconstructive
principles - the ones which will be used latter.

\begin{itemize}
\item \textbf{The limited principle of omniscience},
\textbf{LPO}: \ for each binary sequence $\left( a_{n}\right)
_{n\geq1}$, either  $a_{n}=0$ for all $n\in \mathbb{N}$, or else
there exists $n$ with $a_{n}=1$.

\item
\textbf{The lesser limited principle of omniscience},
\textbf{LLPO}: if $(a_n)_{n\in \mathbb{N}}$ is a binary sequence
containing at most one term equal to 1, then either $a_{2n}=0$ for
all $n\in \mathbb{N}$, or else $a_{2n+1}=0$ for all $n\in
\mathbb{N}$.

\item \textbf{ Markov's principle}, \textbf{MP}: \ For each
binary sequence $\left( a_{n}\right) _{n\geq1}$, if it is impossible
that $a_{n}=0$ for all $n\in \mathbb{N}$, then there exists $n$ with
$a_{n}=1$.
\end{itemize}

\begin{rem}\label{lpo-bhm}
{\rm\textbf{LPO}} is equivalent  to the decidability of equality on
the real number line $\mathbb{R}$.
$$\forall_{x\in \mathbb{R}}\, (x=0 \vee x\neq 0).$$
A detailed constructive study of $\mathbb{R}$ can be found in
\cite{dsb}.
\end{rem}

Within  constructive mathematics,  a statement $P$, as in classical
mathematics, can be disproved by giving a counterexample. However,
it is also possible to give a \emph{Brouwerian counterexample} to
show that the statement is nonconstructive. A Brouwerian
counterexample to a statement $P$ is a constructive proof that $P$
implies some nonconstructive principle, such as, for example,
\textbf{LEM}, and its weaker versions \textbf{LPO}, \textbf{LLPO},
\textbf{MP}. It is not a counterexample in the true sense of the
word - it is just an indication that $P$ does not admit a
constructive proof. More details about nonconstructive principles
and various classical theorems that are not constructively valid can
be found in \cite{ih}.

\subsection{Set with apartness}\label{s-31}

The cornerstones for \textbf{BISH} include the notion of positive
integers, sets and functions. The set $\mathbb{N}$ of positive
numbers is regarded as a basic set, and it is assumed that the
positive numbers have the usual algebraic and order properties,
including  mathematical induction.

 Contrary to the classical case, a set exists only when it is
defined. To define a set  $S$, we have to give a property that
enables us to construct members of $S$, and to describe the equality
$=$ between elements of $S$ -- which is a matter of convention,
except that it must be an  equivalence. A set $(S, =)$ is an
\emph{inhabited} set if we can construct an element of $S$. The
distinction between the notions of a nonempty set and an inhabited
set is a key in constructive set theories. The notion of equality of
different sets is not defined. The only way in which elements of two
different sets can be regarded as equal is by requiring them to be
subsets of a third set. For this reason, the operations of union and
intersection are defined only for sets which are given as subsets of
a given set. There is another problem   to face   when we consider
families of sets that are closed under a suitable operation of
complementation. Following \cite{bdb} ``we do not wish to define
complementation in the terms of negation; but on the other hand,
this seems to be the only method available. The way out of this
awkward position is to have a very flexible notion based on the
concept of \emph{a set with apartness}.''

A property $P$, which is applicable to the elements of a set $S$,
 determines a subset of $S$ denoted by $\{x\in S : P(x)\}$.
 Furthermore, we will be interested only in properties $P(x)$
which are \emph{extensional} in the sense that for all $x_1,x_2\in
S$ with $x_1 = x_2$, $P(x_1)$ and $P(x_2)$ are equivalent.
Informally, it means that ``it does not depend on the particular
description by which $x$ is given to us'', \cite{dbv2}.

An inhabited subset of $S\times S$, or, equivalently, a property
applicable to elements of $S\times S$, is called a \emph{binary
relation} on $S$. In general,  there are many properties that binary
relations may satisfy on a given set. For instance, reflexivity,
symmetry, transitivity, irreflexivity, strong irreflexivity,
co-transitivity  play a role under constructive rules.

In \textbf{CLASS},   equivalence is the natural generalization of
equality.  A theory with equivalence involves
 equivalence and functions, and relations respecting this equivalence.
In constructive mathematics the same   works without difficulty,
\cite{wr}.

 Many sets  come with a binary relation called inequality satisfying certain properties, and
denoted by $\neq$, $\#$ or $\not\backsimeq$. In general, more
computational information is required to distinguish elements of a
set $S$, than to show that elements are equal. Comparing with
\textbf{CLASS}, the situation for inequality
 is more complicated. There are different types of
inequalities (denial inequality, diversity, apartness, tight
apartness - to mention a few), some of them completely independent,
which only in \textbf{CLASS} are equal to one standard inequality.
So, in \textbf{CLASS} the study of the equivalence relation
suffices, but in constructive mathematics, an inequality becomes a
``basic notion in intuitionistic axiomatics''.  Apartness, as a
positive version of inequality, ``is yet another fundamental notion
developed in intuitionism which shows up in computer science,''
\cite{bj}.

\medskip
Let $(S,=)$ be an \emph{inhabited} set. By an
\textbf{\emph{apartness}} on $S$  we mean a binary relation $\# $ on
$S$ which satisfies the axioms of irreflexivity, symmetry and
cotransitivity:
\begin{itemize}
\item[(Ap1)] $\neg(x \# x)$

\item[(Ap2)] $x \# y \ \Rightarrow\ y \# x$,

\item[(Ap3)] $x\# z\ \Rightarrow\ \forall _y\,(x\# y\,\vee\,y\# z).$
\end{itemize}
If $x\# y$, then $x$ and $y$ are different, or distinct.
 Roughly speaking, $x= y$ means that we have a proof that $x$ equals $y$ while
$x\# y$ means that we have a proof that $x$ and $y$ are different.
Therefore, the negation of $x= y$ does not necessarily imply that
$x\# y$ and vice versa: given $x$ and $y$, we may have neither a
proof that $x= y$ nor a proof that $x\# y$.

The negation of apartness is an equivalence
$(\approx)\stackrel{\text def}{=}(\neg\,\#)\,$ called \emph{weak
equality} on $S$.

\begin{rem}\label{aset1}
The statement that every equivalence relation is the negation of
some apartness relation is equivalent to the excluded middle. The
statement that the negation of an equivalence relation is always an
apartness relation is equivalent to the nonconstructive de Morgan
law.
\end{rem}

The apartness on a set $S$ is \emph{tight} if
\begin{itemize}
\item[(Ap4)] $\neg(x \# y) \ \Rightarrow\ x=y$.
\end{itemize}
 Apartness is tight just when
$\approx$ and $=$ are the same, that is $\neg(x \# y) \
\Leftrightarrow\ x=y$.

In some books and papers,  such as \cite{t1}, the term
``preapartness'' is used for an apartness relation, while
``apartness'' means tight apartness. The tight apartness on the real
numbers was introduced by L. E. J. Brouwer in the early 1920s.
Brouwer introduced the notion of apartness as a positive
intuitionistic basic concept. A formal treatment of apartness
relations began with A. Heyting's formalization of elementary
intuitionistic geometry in \cite{ah1}. The intuitionistic
axiomatization of apartness is given in~\cite{ah2}.

\medskip
 By extensionality, we have
\begin{itemize}
\item[(Ap5)] $x \# y \,\wedge\,y = z \ \Rightarrow\, x \# z$,
\end{itemize}
\noindent the equivalent form  of which is

\begin{itemize}
  \leftskip=1.7mm
\item[(Ap5')] $x \# y \,\wedge\,x = x' \,\wedge\, y=y'\ \Rightarrow\,
x' \# y'$.
\end{itemize}

A \emph{set with apartness} $(S,=,\#)$ is the starting point for our
 considerations, and will be simply denoted by $S$.
The existence of an apartness relation on a structure often  gives
rise to an  apartness relation on another structure. For example,
given two sets with apartness $(S,= _S,\#_S)$ and $(T,=_T,\#_T)$, it
is permissible to construct the  set of mappings between them.
Following \cite{dbv2}, \emph{a mapping}
$f:S\rightarrow T$ is an algorithm which produces an element $f(x)$
of $T$ when applied to an element $x$ of $S$, which is extensional,
that is\vspace*{-1mm}
\[
\forall_{x,y\in S}\,(x=_{S}y\  \Rightarrow \ f(x)=_{T}f(y)).\mathbf{\ }%
\]
A mapping $f:S\rightarrow T$ is:

- \, \emph{onto} $S$ or \emph{surjection}: $\forall_{y\in T}\
\exists_{x\in S}\ (y=_Tf(x))$;

- \, \emph{one-one} or \emph{injection}:  $\forall_{x,y\in S} \,
(f(x)=_Tf(y)\ \Rightarrow \ x=_Sy)$;

- \, \emph{bijection} between $S$ and $T$: it is a one-one and onto.

An important property applicable to mapping $f$ is that
of strong extensionality. Namely, a  mapping $f:S\rightarrow T$ is
a \emph{strongly extensional} mapping, or, for short, an
\emph{se-mapping}, if
\[
\forall_{x,y\in S}\, (f(x)\#_{T} f(y)\ \Rightarrow\ x\#_{S} y).
\]
(The strong extensionality of all mappings from $\mathbb{R}$ to
$\mathbb{R}$ implies Markov principle, \textbf{MP}, see
\cite{dbv2}.)

 Furthermore, $f$ is

 -  \emph{apartness injective}, shortly \emph{a-injective}:  $\forall_{x,y\in S}\, (x\#_{S}y\  \Rightarrow \
f(x)\#_{T} f(y))$;

-  \emph{apartness bijective}:  a-injective, se-bijective.

Given two sets with apartness $S$ and $T$ it is permissible to
construct the set of ordered pairs  $(S\times T, =, \#)$ of these
sets defining apartness by $$(s,t)\,\#\, (u,v) \enspace
\stackrel{\rm def}{\Leftrightarrow}\enspace s \,\#_S\, u\,\vee \,
t\, \#_T\, v.$$

\subsubsection{Distinguishing subsets}\label{s-311}

The presence of apartness implies the appearance of different types
of substructures connected to it. Inspired by the constructive
topology with apartness  \cite{dbv2}, we define the relation
$\bowtie$ between an element $x\in S$ and a subset $Y$ of $S$  by
$$x\bowtie Y \ \stackrel{\rm def}{\Leftrightarrow} \ \forall _{y\in Y}  (x\# y).$$
A subset $Y$  of $S$ has two natural complementary subsets:
\emph{the logical complement} of $Y$
$${\neg Y} \stackrel{\rm def}{=} \{x\in S : x\notin Y\}, $$ and \emph{the apartness complement}
 or, shortly,
the \emph{a-complement} of $Y$ $${\sim Y} \,  \stackrel{\rm def}{=}
\, \{x \in S\, : \, x \bowtie Y\}.$$ Denote by $\widetilde{x}$ the
a-complement of the singleton $\{ x\}$. Then it can be easily shown
that $x\in \sim Y$ if and only if $Y\subseteq \widetilde{x}.$

\medskip
If the apartness is not tight we can find subsets $Y$ with ${\sim Y}
\subset {\neg Y}$ as in the following example.
\begin{ex}\label{ex11}
Let  $S=\{a,b,c\}$ be a set with apartness defined by
$\{(a,c),(c,a),(b,c),(c,b)\}$ and let $Y=\{ a\}$. Then the
a-complement $\sim\!\,Y = \{ c\}$ is a proper subset of its logical
complement $\neg Y=\{ b,c\}$.
\end{ex}
For a tight apartness, the two complements are constructive
counterparts of the classical complement. In general, we have ${\sim
Y} \subseteq {\neg Y}.$ However, even for a tight apartness, the
converse inclusion entails the Markov principle, \textbf{MP}. This
result illustrates a main feature of constructive mathematics:
classically equivalent notions could be no longer equivalent
constructively. For which  type of subset  of a set with apartness
do we  have equality between its two complements?
 It turns out that the answer initiated a development of \emph{order theory} for sets and
 semigroups with apartness

\medskip
The complements  are used for the classification of subsets of a
given set. A subset  $Y$  of $S$ is
\begin{itemize}
\itemsep=0.95pt
  \item   \emph{ a detachable} subset  in $S$  or, in short, a
\emph{d-subset} in $S$ if
$$
\forall _{ x\in S} \, (x \in Y \vee x \in {\neg} Y);
$$
\item \emph{a strongly detachable} subset of $S$, shortly \emph{an
sd-subset of} $S$, if
$$
\forall _{ x\in S} \, (x \in Y \vee x \in {\sim} Y),
$$
\item \emph{a quasi-detachable} subset of $S$, shortly \emph{a
qd-subset of} $S$, if
$$
\forall _{ x\in S} \,\forall _{ y\in Y} \, (x \in Y \vee x \# y).
$$
\end{itemize}

The relations between detachable, strongly detachable and
quasi-detachable subsets are  partially described in \cite{mmscr},
Proposition 2.1. A  description of the relationships between those
subsets of set with apartness, which, in turn, justifies the
constructive order theory for sets and semigroups with apartness we
develop, is given in the next theorem which is one of the main
results of this paper.
\begin{teo}\label{senegcom}
Let $Y$ be a subset of $S$. Then:
\begin{itemize}
\item[\emph{(i)}]  Any sd-subset is a qd-subset of $S$. The converse implication entails  {\rm{\textbf{LPO
}}}.
\item[\emph{(ii)}] Any qd-subset $Y$ of $S$ satisfies ${\sim Y}={\neg
Y}$.
\item[\emph{(iii)}]  If any qd-subset is a d-subset, then
{\rm{\textbf{LPO}}} holds.
\item[\emph{(iv)}] If any d-subset is a qd-subset, then {\rm{\textbf{MP}}} holds.
\item[\emph{(v)}]  Any sd-subset is a d-subset of $S$. The converse implication
entails {\rm{\textbf{MP}}}.
\item[\emph{(vi)}] If any   subset of a set with apartness $S$ is a qd-subset, then
{\rm{\textbf{LPO}}} holds.
\end{itemize}
\end{teo}
\begin{proof} (i). \ Let $Y$ be an sd-subset of $S$. Then, applying
the definition and  logical axiom we have
\begin{eqnarray*}
\forall _{ x\in S} \, (x \in Y \vee x \in {\sim} Y)\
&\Leftrightarrow \ &
\forall _{ x\in S} \,(x \in Y \, \vee \,  \forall _{ y\in Y} (x\# y))\\
 &\Rightarrow & \forall _{ x\in S} \,\forall _{ y\in Y} \, (x \in Y \vee x \# y).
\end{eqnarray*}

In order to prove the second part of this statement, we consider the
real number set $\mathbb{R}$ with the usual (tight) apartness and
the subset $Y=\widetilde{0}.$ Then, for each real number $x$ and for
each $y\in Y$ it follows, from the co-transitivity of $\#$, either
$y\# x$ or $x\# 0 $, that is, either $x\in Y$ or $x\# y$.
Consequently, $Y$ is a qd-subset of $\mathbb{R}.$ On the other hand,
if $Y$ is an sd-subset of $\mathbb{R},$ then for each $x\in
\mathbb{R}$, either $x\in Y$ or $x\in \sim Y.$ In the former case,
$x\# 0$ and in the latter $x=0$, hence \textbf{LPO} holds.

\smallskip
(ii). \ Let $Y$ be a qd-subset, and let $a\in \neg Y$. By assumption
we have
$$\forall _{ x\in S} \,\forall _{ y\in Y} \, (x
\in Y \vee x \# y),$$ so substituting $a$ for $x$, we get
$\forall_{y\in Y} \, (a\in Y \vee a\# y)$, and since, by assumption,
$\neg(a\in Y)$, it follows that $a\# y$ for all $y\in Y$. Hence
$a\in \sim Y$. See also \cite{mmscr}.

\smallskip
(iii). \ Let $S$ be the real number set $\mathbb{R}$ with the usual
 apartness $\#$. As in the proof of (i), consider the qd-subset $\widetilde{0}$
of $\mathbb{R}$. If $\widetilde{0}$ is a d-subset of $\mathbb{R}$,
then $x\in \widetilde{0}$ or $\neg(x\in \widetilde{0})$, for all
real numbers $x$. In the latter case $\neg(x\# 0)$,  which is
equivalent to $x=0.$ Thus we obtain the property $\forall_{ x\in
\mathbb{R}}\; (x\# 0\vee x=0)$ which, in turn, is equivalent to
\textbf{LPO}.

\smallskip
(iv).  \ Consider a real number $a$ with $\neg (a=0)$ and let $S$ be
the set $\{ 0,a\}$ endowed with the usual apartness of $\mathbb{R}$.
For $Y=\{ 0\}$, since $0\in Y$ and $a\in \neg Y,$ it follows that
$Y$ is  a d-subset of $S$. On the other hand, if $Y$ is a qd-subset
of $S$, then $a\# 0$. It follows that for any real number with $\neg
(a=0)$, $a\# 0$ which entails the Markov Principle, \textbf{MP}.

\smallskip
(v). \ The first part follows immediately from (i), (ii) and the
definition of d-subsets. The converse follows  from (i) and (iv).

\smallskip
(vi). \ Consider again $\mathbb{R}$ with the usual apartness and
define $Y=\{ 0\} .$ If $Y$ is a qd-subset of $\mathbb{R}$, then for
all $x\in \mathbb{R}$ we have  $x=0$ or $x\# 0$, hence \textbf{LPO}
holds.
\end{proof}

If the apartness is not tight, we can find subsets which are not
qd-subsets, let alone sd-subsets. To show this, let us consider the
set $S=\{ a,b,c\}$ with the apartness defined in Example \ref{ex11}
and define $Y=\{ a\}.$ Then $Y$ is not a qd-subset of $S$. If we
work with a tight apartness, although vacuously true in classical
mathematics, the properties of detachability are not automatically
satisfied in \textbf{BISH}. The Brouwerian examples from Theorem
\ref{senegcom} motivate the use of qd-subsets. Constructive
mathematics brings to the light some notions which are invisible to
the classical eye (here, the three notions of detachability).

\subsubsection{Co-quasiorders}\label{s-312}

Let $(S\times S,=,\#)$ be a set with apartness. An inhabited subset
of $S\times S$, or, equivalently, a property applicable to the
elements of $S\times S$, is called a \emph{binary relation} on $S$.
Let $\alpha$ be a relation on $S$.
 Then $$(a,b)\ \bowtie \ \alpha \  \Leftrightarrow \ \forall
_{(x,y)\in\alpha }\ ((a,b)\, \#\, (x,y)),$$ \noindent for any
$(a,b)\in S\times S$. The apartness complement of $\alpha$ is the
relation
$$\mathbf{\sim \alpha } \, =\, \{(x,y)
\in S\times S : (x,y) \bowtie \alpha\}.$$ In general, we have
$\sim\alpha \subseteq \neg \alpha$, which is shown by the following
example.
\begin{ex}\label{ex1}
Let  $S=\{a,b,c\}$ be a set with apartness defined by
$\{(a,c),(c,a),(b,c),(c,b)\}$. Let  $\alpha =\{(a,c),(c,a)\}$ be a
  relation on $S$. Its a-complement
$$\sim\!\,\alpha = \{(a,a),(b,b),(c,c),(a,b), (b,a)\}$$
is a proper subset of its logical complement $\neg\alpha$.
\end{ex}

The relation $\alpha$ defined on a set with apartness $S$ is
\begin{itemize}
\itemsep=0.95pt
\item irreflexive if $\forall_{x\in S} \,\neg((x,x)\in \alpha)$;

\item strongly irreflexive if  $(x,y)\in \alpha \ \Rightarrow\ x\#y$;

\item co-transitive if $(x,y)\in\alpha \ \Rightarrow\ \forall _{z\in S}\,((x,z)\in\alpha \,\vee\,(z,y)\in\alpha) .$
\end{itemize}

It is easy to check that a strongly irreflexive relation is also
irreflexive. For a tight apartness, the two notions of irreflexivity
are classically equivalent but not so constructively. More
precisely, if each irreflexive relation were strongly irreflexive
then \textbf{MP} would hold.

In the constructive order theory, the notion of co-transitivity,
that is the property that for every pair of related elements, any
other element is related to one of the original elements in the same
order as the original pair is a constructive counterpart to
classical transitivity, \cite{cmr1}.

\begin{lem} \label{reflexiveconsistent}
 Let $\alpha$ be a relation on $S$. Then:
 \begin{enumerate}
 \itemsep=0.95pt
 \item[\emph{(i)}] $\alpha$ is strongly irreflexive if and only if $\sim\!\,\alpha$ is reflexive;
 \item[\emph{(ii)}] if $\alpha$ is reflexive then $\sim\!\,\alpha$ is strongly irreflexive;
 \item[\emph{(iii)}] if $\alpha$ is symmetric then $\sim\!\,\alpha$ is
 symmetric;
 \item[\emph{(iv)}] if $\alpha$ is co-transitive then $\sim\!\,\alpha$
 is transitive.
\end{enumerate}
\end{lem}

\begin{proof} (i). \ Let $\alpha$ be a strongly irreflexive  relation on $S$.
For each $a\in S$, it can be easily proved that $(a,a)\# (x,y)$ for
all $(x,y)\in \alpha$.

Let $\sim\!\,\alpha$ be reflexive, that is $(x,x)\in\sim\!\,\alpha$,
for any $x\in S$. On the other hand, the definition of the
a-complement implies $(x,y)\# (x,x)$ for any $(x,y)\in \alpha$. So,
$x\# x$ or $x\# y$. Thus, $x\# y$, that is, $\alpha$ is strongly
irreflexive.

\smallskip
(ii). \  Let $\alpha$ be reflexive. Let $(x,y)$ be an element of
$\sim\alpha$. Since $\alpha$ is reflexive, $(y,y)\in\alpha$ hence
$(x,y)\# (y,y)$ which implies $x\# y.$ Consequently, $\sim\alpha$ is
strongly irreflexive.

\smallskip
(iii). \ If $\alpha$ is symmetric, then
\begin{eqnarray*}
(x,y)\in\sim\!\,\alpha \ &\Leftrightarrow & \ \forall _{(a,b)\in \alpha} \, ((x,y)\#(a,b))\\
&\Rightarrow  & \ \forall _{(b,a)\in \alpha} \, ((x,y)\#(b,a))\\
 &\Rightarrow & \ \forall _{(b,a)\in \alpha} \, (x\# b \ \vee \ y\# a)\\
 &\Rightarrow & \ \forall _{(a,b)\in \alpha} \, ((y,x)\# (a,b))\\
&\Leftrightarrow & \ (y,x)\in\sim\!\,\alpha.
\end{eqnarray*}

\smallskip
(iv). \ If $(x,y)\in\sim\alpha$ and $(y,z)\in\sim\alpha$, then, by
the definition of $\sim\alpha$, we have that $(x,y)\bowtie \alpha$
and $(y,z)\bowtie \alpha$. For an element $(a,b)\in\alpha$, by
co-transitivity of $\alpha$, we have $(a,x)\in\alpha$ or
$(x,y)\in\alpha$ or $(y,z)\in\alpha$ or $(z,b)\in\alpha$. Thus
$(a,x)\in\alpha$ or $(z,b)\in\alpha$, which implies that $a\# x$ or
$b\# z$, that is $(x,z)\#(a,b)$. So, $(x,z)\bowtie \alpha$ and
$(x,z)\in\ \sim\alpha$. Therefore, $\sim\alpha$ is transitive.
\end{proof}

\begin{rem}\label{dbv-b}
As it is shown in Lemma~\ref{reflexiveconsistent}, it can be proved
that the logical complement of each co-transitive relation is
transitive. However, if the logical complement of any transitive
relation were co-transitive, then \textbf{LLPO} would hold.

\medskip
To prove this, let us consider the strict order $<$ on
the real number line and assume that its logical complement $\geq$
is cotransitive. Then, taking into account that $x\geq x$, it
follows that
$$\forall x\in {\bf R}\; (x\geq 0\vee 0\geq x)$$ which,
in turn, is equivalent to {\bf LLPO}. This Brouwerian example can be
found in books on constructive mathematics, see, for example,
\cite{dbv2}.
\end{rem}

The apartness complement $\sim\alpha$ of a relation $\alpha$ of $S$
can be transitive without assuming co-transitivity of $\alpha$. So,
the converse statement from Lemma~\ref{reflexiveconsistent}(iv), in
general, is not true.

\begin{ex} \label{cex-1}
Let $(S, = , \# )$ be a set with apartness defined in
Example~\ref{ex1}.

\smallskip
(1.) A strongly  irreflexive (symmetric) relation $\alpha
=\{(a,c),(c,a)\}$, which is not co-transitive has the a-complement
$$\sim \alpha = \{(a,a),(b,b),(c,c),(a,b), (b,a)\}$$
which is transitive.

\smallskip
(2.)  A strongly irreflexive (nonsymmetric)  relation $\alpha
=\{(a,c),(c,a), (b,c)\}$, which is not co-transitive, has the
a-complement
$$\sim\alpha = \{(a,a),(b,b),(c,c),(a,b), (b,a)\}$$
which is transitive.
\end{ex}

\begin{rem}
The Brouwerian example from Remark~\ref{dbv-b} also shows that, even
for a tight apartness, a relation whose co-transitivity cannot be
proved constructively might have a transitive a-complement. To prove
this, we need only observe that the a-complement of the relation
$\geq$ is $<$.
\end{rem}

A relation $\tau$ defined on a set with apartness $S$ is a
\begin{itemize}
\item \emph{weak co-quasiorder} if it is irreflexive and co-transitive,
\item \emph{co-quasiorder } if it is strongly irreflexive and co-transitive.
\end{itemize}

\begin{rem}\label{r2-bhm}
``One  might expect that the splitting of notions leads to an
enormous proliferation of results in the various parts of
constructive mathematics when compared with their classical
counterparts. In particular, usually only very few constructive
versions of a classical notion are worth developing since other
variants do not lead to a mathematically satisfactory theory,''
\cite{t1}.

Even if the two classically   (but not constructively) equivalent
variants of a co-quasiorder are constructive counterparts of a
quasiorder in the case of (a tight) apartness, the stronger variant,
co-quasiorder, is, of course, the most appropriate for a
constructive development of the theory of semigroups with apartness
we develop, which will be evident in the continuation of this paper.
The weaker variant, that is, weak co-quasiorder, could be relevant
in analysis.
\end{rem}

As in Example~\ref{ex1} the a-complement of a relation can be a
proper subset of its logical complement. If the relation in question
is a co-quasiorder, then we have the following important properties.

\begin{prop} \label{senegcompl}
 Let  $\tau$ be a co-quasiorder on $S$. Then:
\begin{itemize}
\item[\emph{(i)}]  $\tau$ is a qd-subset of $S\times S$;
\item[\emph{(ii)}] $\sim\!\,\tau = \neg\,\tau$.
\end{itemize}
\end{prop}
\begin{proof}
(i). \ Let $(x,y)\in S\times S$. Then, for all $(a,b)\in\tau$,
\begin{eqnarray*}
    a\tau x \vee x\tau b  &\ \Rightarrow\ & a\tau x \vee x\tau y \vee y\tau b\\
  & \ \Rightarrow\ & a\# x \vee x\tau y\vee y\# b \\
  & \ \Rightarrow\ & (a,b)\# (x,y) \vee x\tau y,
\end{eqnarray*}
that is, $\tau$ is a qd-subset.

\vspace{2mm}
(ii). \ It follows from (i) and Theorem~\ref{senegcom}(ii).
\end{proof}
See also \cite{mmscr}.

 The co-quasiorder is one of the main building
blocks for the order theory of semigroups with apartness we develop.

\smallskip
In general, to describe the relation we have to determine which
ordered pairs belong to $\tau $, that is, we have to determine
$a\tau$ and $\tau a$, the left and the right $\tau$-class of each
element $a$ from $S$. That is the way to connect (in \textbf{CLASS}
and in \textbf{BISH} as well) a relation defined on a given set with
certain subsets of the set. Starting from an sd-subset $T$ of $S$,
we are able to construct co-quasiorders as follows.

\begin{lem} \label{lem3.1}
Let $T$ be an sd-subset of a set with apartness $S$. Then, the
relation $\tau$ on $S$, defined by
$$
(a,b)\in \tau \ \stackrel{\rm def}{\Leftrightarrow} \ a \in \sim
T\,\wedge \, b \in T,
$$
is a  co-quasiorder on $S$.
\end{lem}

\begin{proof}
 Let  $(a,b)\in \tau$,  that is $a \in \sim T$ and $b\in T$,
and  let $x\in S$. By the assumption, $T$ is an sd-subset, so we
have $x\in T$ or $x\in \sim T$. If $x\in T$, then, by the definition
of $\tau$, we have $(a,x)\in  \tau$. Similarly, if $x\in \sim T$,
then $(x,b)\in \tau$. Thus, co-transitivity of $\tau$ is proved. By
the definition of $\tau$, the strong irreflexivity follows
immediately. Thus, $\tau$ is a co-quasiorder on $S$.
\end{proof}

\begin{ex}\label{bhm-e3}
Let $S = \{a,b,c,d,e\}$ be a  set with the diagonal
$$\triangle_S = \{ (a,a), (b,b), (c,c), (d,d) , (e,e)\}$$
 as the equality relation. If we denote  by
 $K$ the set $\triangle_S \cup \{(a,b),(b,a)\}$, then we can define an apartness $\#\,$ on $S$
 to be  $(S\times S) \setminus  K$. Thus, $(S, = , \#\,)$ is a set with apartness. The
 relation $\tau \subseteq S \times S$, defined by
$$\tau  = \{(c,a),(c,b),(d,a),(d,b),(d,c),(e,a),(e,b),(e,c), (e,d)\},$$
 is a co-quasiorder  on $S$. (Left) $\tau$-classes of $S$ are:
$a\tau=b\tau = \emptyset$, $c\tau = \{a,b\}$, $d\tau = \{a,b,c\}$,
$e\tau = \{a,b,c,d\}$. It can be easily checked that all those
$\tau$-classes are sd-subsets of $S$.
\end{ex}

Generally speaking, for  a co-quasiorder defined on a set with
apartness we can not prove that its left and/or right classes are
d-subsets or sd-subsets. More precisely, we can prove the following
result.

\begin{prop}\label{p11-bhm}
Let $\tau$ be a co-quasiorder. Then:
\begin{itemize}
\item[\emph{(i)}] if $a\tau $ is a d-subset of $S$ for any $a\in S$,
then \rm{\textbf{LPO}} holds;
\item[\emph{(ii)}] if $a\tau $ is an sd-subset of $S$ for any $a\in S$,
then \rm{\textbf{LPO}} holds.
\end{itemize}
\end{prop}
\begin{proof}
(i). \ Similar to the proof of Theorem~\ref{senegcom}(iii). It
suffices to let $\tau$ be the usual apartness on the real number set
and $a=0.$

\vspace{1.8mm}
(ii). \ We can use the same example as above and apply
Theorem~\ref{senegcom}(i).
\end{proof}

Having in mind what is  just proved, we cannot expect to prove Lemma
3.2 from \cite{cmr2}, as stated in \cite{cmr2}, with d-subsets or
sd-subsets without the Constant Domain Axiom, \textbf{CDA}.

\subsubsection{Intuitionistic logic of constant domains CD as a background}

Following \cite{fa},  the intuitionistic logic of constant domains
\textbf{CD } arises from a very natural Kripke-style semantics,
which was proposed in \cite{ag}  as a philosophically plausible
interpretation of intuitionistic logic.  \textbf{\textbf{CD}}  can
be formalized as intuitionistic logic extended with from the
classical algebra point of view pretty strong principle, the
Constant Domain Axiom, CDA,
$$
\Vdash \forall _x\, (P\vee R(x))\, \rightarrow \  (P\vee \forall
_x\,R(x)),
$$
where $x$ is not a free variable of $P$. The intermediate logic
 obtained in this way, as it is pointed out in \cite{fa}, further
 proves intuitionistically as well as classically valid theorems, yet they
often possess a strong constructive flavour.

\medskip
From a given co-quasiorder $\tau$, with \textbf{CD} as a logical
background, we are able to prove the connection of its classes  with
sd-subsets of $S$.

\begin{lem} \label{lem3.2}
Let $\tau$ be a co-quasiorder on a set $S$. Then $a\tau$
(respectively $\tau a$) is an sd-subset of $S$, such that $a\bowtie
a\tau$ (respectively $a\bowtie\tau a$), for any $a\in S$. Moreover,
if $(a,b)\in \tau$, then $a\tau \cup \tau b = S$ is true for all
$a,b\in S$.
\end{lem}

\begin{proof}
See \cite{cmr2}.
\end{proof}

Lemma~\ref{lem3.2}, due to Proposition~\ref{p11-bhm}, cannot be
proved outside \textbf{CD} as logical background.

\begin{rem}
In intuitionistic logic of constant domain \textbf{CD}, the notions
of sd-subset and qd-subset coincide.
\end{rem}

\subsubsection{QSP  for sets with apartness}\label{s-313}

The  Quotient Structure Problem, \textbf{QSP},  is one of the very
first problems which has to be considered for any structure with
apartness. The solutions of the \textbf{QSP} problem for sets and
semigroups with apartness was given in \cite{cmr1}. Those results
are improved in \cite{mmscr}. In what follows,  we achieve a little
progress in that direction. Theorem~\ref{BHM-1}, the key theorem for
the \textbf{QSP}'s solution generalizes the similar ones from
\cite{cmr1}, \cite{mmscr}. In addition, as a generalization of the
Theorem~\ref{cmrth2}, the first apartness isomorphism theorem, the
new Theorem~\ref{basicfactor}, that we call the second apartness
isomorphism theorem for sets with apartness is formulated and
proved.

\medskip
The quotient structures are not part of \textbf{BISH}. A quotient
structure does not have, in general, a natural apartness relation.
For most purposes, we overcome this problem using a
\textbf{\emph{co-equivalence}}--symmetric co-quasiorder--instead of
an equivalence. Existing properties of a co-equivalence guarantee
its a-complement  is an equivalence as well as the quotient set of
that equivalence will inherit an apartness. The following notion
will be necessary. For any two relations  $\alpha$ and
$\beta$ on $S$ we can   say that $\mathbf{\alpha}$ \emph{defines an
apartness on} ${S/\beta}$  if we have\smallskip

(Ap6) \ \  $x\beta\,\# \,y\beta \ \stackrel{\rm
def}{\Leftrightarrow}\ (x,y)\in\alpha$.

If in addition $\alpha$ is a co-quasiorder and $\beta$
is an equivalence, then (Ap6) implies

(Ap6') \ \ $((x,a)\in\beta\,\wedge\,(y,b)\in\beta) \ \Rightarrow \ (
(x,y)\in\alpha \,\Leftrightarrow\, (a,b)\in\alpha).$

\vspace{1.8mm}
Indeed, let $\alpha$ be a co-quasiorder and $\beta$ an equivalence
on $S$ such that  $\mathbf{\alpha}$ \emph{defines an apartness on}
${S/\beta}$. Let $(x,a), (y,b) \in\beta$, i.e. $a\in x\beta$
and $b\in y\beta$, which, by the assumption, gives $a\beta = x\beta$
and $b\beta = y\beta$. If $(x,y)\in \alpha$, then, by (Ap6), $x\beta
\,\#\, y\beta$, which, by (Ap5'), gives $a\beta \,\#\, b\beta$. By
(Ap6) we have $(a,b)\in\alpha$. In a similar manner, starting from
$(a,b)\in\alpha$ we can conclude $(x,y)\in \alpha$.

 The next  theorem is the key for the solution of
\textbf{\textbf{QSP}} for sets with apartness. It generalizes the
results from \cite{cmr1}, \cite{mmscr}.

\begin{teo}\label{BHM-1}
Let $S$ be a set with apartness. Then:
\begin{itemize}
\itemsep=0.95pt
\item[\emph{(i)}] Let $\varepsilon$ be an equivalence, and $\kappa$ a co-equivalence on $S$. Then, $\kappa$
  defines an apartness  on the factor set $S/\varepsilon$ if and only if $\varepsilon\cap\kappa =
  \emptyset$.
\item[\emph{(ii)}]  The quotient mapping $\pi : S \rightarrow S/\varepsilon$, defined by
$\pi (x)  = x\varepsilon$, is an onto se-mapping.
\end{itemize}
\end{teo}
\begin{proof}
(i). \ \ Let $x,y\in S$ and assume that $(x,y)\in
\varepsilon\cap\kappa$. Then $(x,y)\in \varepsilon$ and $(y,y)\in
\varepsilon$, which, by  (Ap6') and  and $(x,y)\in\kappa$, gives
$(y,y)\in \kappa$, which is impossible. Thus, $\varepsilon\cap\kappa
= \emptyset$.

\medskip
Let $(x,a), (y,b)\in \varepsilon$ and $(x,y)\in\kappa$. Then, by
co-transitivity of $\kappa$ and by assumption, we have
\begin{eqnarray*}
(x,y)\in \kappa &\ \Rightarrow\ & (x,a)\in \kappa \vee (a,y)\in
\kappa\\
 &\ \Rightarrow\ & (x,a)\in \kappa \vee (a,b)\in \kappa \vee (b,y)\in\kappa\\
 &\ \Rightarrow\ &  (a,b)\in \kappa.
\end{eqnarray*}

(ii). \ Let $\pi(x)\#\pi(y)$, that is $x\varepsilon \#
y\varepsilon$, which, by (i), means that $(x,y)\in \kappa$. Then, by
the strong irreflexivity of $\kappa$, we have $x\# y$. So $\pi$ is
an se-mapping.

\medskip
Let $a\varepsilon\in S/\varepsilon$ and $x\in a\varepsilon$. Then
$(a,x)\in\varepsilon$, i.e. $a\varepsilon = x\varepsilon$, which
implies that $a\varepsilon = x\varepsilon=\pi(x)$. Thus $\pi$ is an
onto mapping.
\end{proof}

\begin{cor}\label{mhb-c32}
If $\kappa$ is a co-equivalence on $S$, then the relation
${\sim}\kappa(={\neg} \kappa)$ is an equivalence on $S$, and
$\kappa$ defines an apartness on  $S/\,\sim\kappa$.
\end{cor}
\begin{proof}
By Lemma~\ref{reflexiveconsistent}, $\sim\kappa$ is an equivalence,
by  Proposition~\ref{senegcompl}, $(\sim\kappa) = (\neg \kappa)$,
and, by Theorem~\ref{BHM-1}, $\kappa$ defines an apartness on
$S/\sim\kappa$.
\end{proof}

Let $f : S\to T$ be an se-mapping between sets with apartness. Then
the relation%
\[
\mathrm{coker}\,f \stackrel{\rm def}{=} \{(x,y)\in S\times S:f(x)\#
f(y)\}
\]
defined on $S$ is called the  \emph{co-kernel }of $f$. Now,
\emph{the first apartness isomorphism theorem } for sets with
apartness follows.

\begin{teo} \label{cmrth2}
Let $f : S\to T$ be an se-mapping between sets with apartness. Then
\begin{itemize}
\itemsep=0.95pt
\item[\emph{(i)}] the co-kernel of $f$ is a co-equivalence on $S$
 which defines an apartness on  $S/\ker\,f$;
\item[\emph{(ii)}]  the mapping $\theta:S/\ker
\,f\to T$, defined by $\theta(x(\ker\,f))=f(x)$, is a one-one,
a-injective se-mapping such that $f=\theta\circ\pi$;
\item[\emph{(iii)}] if $f$ maps $S$ onto $T$, then
$\theta$ is an apartness bijection.
\end{itemize}
\end{teo}

\begin{proof}
See \cite{mmscr}.
\end{proof}

Now,  \emph{the second apartness isomorphism theorem},  a
generalised version of Theorem~\ref{cmrth2}, for sets with apartness
follows.

\begin{teo} \label{basicfactor}
  Let $f:S\to T$ be a mapping between sets with apartness,  and let
  $\kappa$ be a co-equivalence on $S$ such that $\kappa\cap \mathrm{ker}\,f=\emptyset$. Then:
\begin{itemize}
  \item[\emph{(i)}]  $\kappa$ defines apartness on factor set $S/\mathrm{ker}\,f$;
\eject
  \item[\emph{(ii)}] the projection  $\pi: S\to S/\mathrm{ker}\,f$ defined by $\pi (x)= x(\mathrm{ker}\,f)$ is an onto se-mapping;

  \item[\emph{(iii)}] the mapping $f$ induces a one-one mapping $\theta:S/\mathrm{ker}\,f\to T$ given by $\theta(x(\mathrm{ker}\,f))=f(x)$, and $f=\theta\circ\pi$;

  \item[\emph{(iv)}] $\theta$ is an se-mapping if and only if $\mathrm{coker}\,f\subseteq\kappa$;

  \item[\emph{(v)}]  $\theta$ is a-injective if and only if $\kappa\subseteq \mathrm{coker}\,f$.
  \end{itemize}
\end{teo}

\begin{proof}
\indent (i). \  It follows from Theorem~\ref{BHM-1}(i).

(ii). \ It follows from Theorem~\ref{BHM-1}(ii).

(iii). \ This was shown in Theorem~\ref{cmrth2}.

(iv). \ Let $\theta$ be an se-mapping. Let $(x,y)\in
\mathrm{\mathrm{\mathrm{coker}}}\, f$ for some $x,y\in S$. Then, by
definition of $\mathrm{coker}\,f$ and $\theta$, the assumption and
(i), we have
\begin{eqnarray*}
f(x)\# f(y)  \ &\Leftrightarrow & \ \theta(x(\mathrm{ker}\,f))\#
\theta(y(\mathrm{ker}\,f))\\
&\Rightarrow  & \  x(\mathrm{ker}\,f) \# y(\mathrm{ker}\,f)\\
&\Leftrightarrow & \ (x, y)\in \kappa .
\end{eqnarray*}

Conversely, let $\mathrm{coker}\,f\subseteq\kappa$. By  assumption,
(i), (iii) and the definitions of $\theta$ and $\mathrm{coker}\,f$, we
have
\begin{eqnarray*}
\theta(x(\mathrm{ker}\,f))\# \theta(y(\mathrm{ker}\,f)) \
&\Leftrightarrow & \ f(x)\# f(y)\\
&\Leftrightarrow & \ (x,y)\in \mathrm{coker}\,f\\
&\Rightarrow  & \ (x,y)\in \kappa\\
&\Leftrightarrow & \ x(\mathrm{ker}\,f)\# y(\mathrm{ker}\,f).
\end{eqnarray*}

(v). \ Let $\theta$ be a-injective, and let $(x,y)\in \kappa$. Then,
by (iii), we have
\begin{eqnarray*}
x(\mathrm{ker}\,f) \# y(\mathrm{ker}\,f) \ &\Rightarrow  & \
\theta(x(\mathrm{ker}\,f))\# \theta(y(\mathrm{ker}\,f))\\
  & \Leftrightarrow & \ f(x)\# f(y) \\
&\Leftrightarrow & \ (x,y)\in \mathrm{coker}\,f.
\end{eqnarray*}

Conversely, let $\kappa\subseteq \mathrm{coker}\,f$. Then
\begin{eqnarray*}
x(\mathrm{ker}\,f) \# y(\mathrm{ker}\,f) \  & \Leftrightarrow & \
(x,y)\in \kappa\\
&\Rightarrow  & \ (x,y)\in \mathrm{coker}\,f\\
&\Leftrightarrow & \ f(x)\# f(y) \\
&\Leftrightarrow & \ \theta(x(\mathrm{ker}\,f))\#
\theta(y(\mathrm{ker}\,f)).
\end{eqnarray*}

\vspace*{-7mm}
\end{proof}

\begin{cor} \label{co11-bhm}
  Let $f:S\to T$ be a mapping between sets with apartness,  and let
  $\kappa$ be a co-equivalence on $S$ such that $\kappa\cap \mathrm{ker}\,f=\emptyset$. Then:
 \begin{itemize}
  \item[\emph{(i)}]  $f$ is an se-mapping if and only if  $\mathrm{coker}\,f$ is strongly irreflexive;
  \eject
  \item[\emph{(ii)}] if
$\theta:S/\mathrm{ker}\,f\to T$, defined by
  $\theta(x(\mathrm{ker}\,f))=f(x)$, is an se-mapping, then $f$ is an
  se-mapping too.
 \end{itemize}
 \end{cor}
\begin{proof}
\indent (i). \ Let $f$ be an se-mapping. Then, by Theorem~\ref{cmrth2},
$\mathrm{coker}\,f$ is strongly irreflexive. The converse is almost
obvious.

\smallskip
(ii). \ If $\theta$ is an se-mapping then, by
Theorem~\ref{basicfactor}(iv), we have that
$\mathrm{coker}\,f\subseteq\kappa$. So, the strong irreflexivity of
$\kappa$ implies the  strong irreflexivity of $\mathrm{coker}\,f$,
which, by (i), implies $f$ is an se-mapping.
\end{proof}

\begin{rem}
The notion of co-quasiorder first appeared in \cite{dar-96}.
However, let us mention that the results reported from
\cite{dar-96}: Theorem 0.4, Lemma 0.4.1, Lemma 0.4.2, Theorem 0.5
and Corollary 0.5.1 (pages 10-11 in \cite{dar-02}) are not correct.
Indeed, the mentioned filled product is not associative in general.
The notion of co-equivalence, i.e. a symmetric co-quasiorder, first
appeared in \cite{mb-dar}.
\end{rem}

\subsection{Semigroups with apartness}\label{s-32}

Given a set with apartness $(S,=,\#)$, the tuple $(S, =, \#, \,
\cdot )$ is a \emph{semigroup with apartness} if the binary
operation $\cdot$ is associative
\begin{itemize}
\item[(A)] $ \enspace\quad \forall _{a,b,c \in S} \ [(a\cdot b)\cdot c\, =\, a\cdot (b\cdot c)]$,
\end{itemize}
and strongly extensional
\begin{itemize}
\item[(S)] $ \enspace\quad \forall _{a,b,x,y \in S}\ (a\cdot x \#\, b\cdot y
\Rightarrow(a \#\, b  \,\vee \,  x \#\, y))$.
\end{itemize}

\noindent As usual, we are going to write $ab$ instead of $a\cdot
b$. For example, for a given set with apartness $A$ we can construct
a semigroup with apartness $S= A^A$ in the following way.

\begin{teo}\label{evid}
 Let $S$ be the set of all se-functions from $A$ to $A$ with the
standard equality $=$
$$f=\, g \ \Leftrightarrow \  \forall _{x\in A} \, (f(x)=\, g(x))$$ and
apartness $$f\#\, g \  \Leftrightarrow\ \exists _{x\in A} \,
(f(x)\#\,  g(x)).$$ Then $(S, =, \# , \circ )$ is a semigroup with
respect to the binary operation $\circ$ of composition of functions.
\end{teo}
\begin{proof}
See \cite{cmr2}.
\end{proof}
Until the end of of this paper, we adopt the convention that
\emph{semigroup} means \emph{semigroup with apartness}.
 Apartness  from  Theorem~\ref{evid} does not have to be
tight, \cite{cmr1}.

\medskip
Let $S$ and $T$ be semigroups with apartness. A mapping $f : S
\rightarrow T$ is a homomorphism if
$$\forall_{x,y\in S} \, (f(xy) = f(x)f(y)).$$
 A homomorphism $f$ is
\begin{itemize}
\itemsep=0.95pt
\item an \emph{se-embedding} if it is one-one and strongly extensional;
\eject
\item an \emph{apartness embedding} if it is a-injective se-embedding;

\item an \emph{apartness isomorphism} if it is apartness bijection and
se-homomorphism.
\end{itemize}

Within \textbf{CLASS}, the semigroups can be viewed, historically,
as an algebraic abstraction of the properties of the composition of
 transformations on a set. Cayley's theorem for semigroups (which
can be seen as an extension of the celebrated Cayley's theorem on
groups) stated that every semigroup can be embedded in a semigroup
of all self-maps on a set. As a consequence of the
Theorem~\ref{evid}, we can formulate \emph{the constructive Cayley's
theorem for semigroups with apartness} as follows.

\begin{teo}\label{Cayley}
Every semigroup with apartness se-embeds into the semigroup of all
strongly extensional self-maps on a set.
\end{teo}

\begin{proof}
See \cite{cmr2}.
\end{proof}

\begin{rem}
Following \cite{fr1}, the term ``constructive theorem'' refers to a
theorem with constructive proof. A classical theorem that is proven
in a constructive manner is a constructive theorem.
\end{rem}

It is a pretty common point of view that classical theorem becomes
more enlightening when it is seen from the constructive viewpoint.
On the other hand, it can not be said that the theory of
constructive semigroups with apartness aims at revising the whole
classical framework in nature.

\subsubsection{Co-quasiorders defined on a semigroup}\label{s-321}

We are going to encounter sd-subsets or sd-subsemigroups which have
some of the properties mentioned in Section~\ref{s-211}. A strongly
detachable convex (respectively completely isolated) subsemigroup of
$S$ is called, in  short, an \emph{sd-convex} (respectively
sd-completely isolated) subsemigroup of $S$. Similarly, there are
sd-convex and sd-completely isolated ideals of $S$.

\begin{lem}\label{lem3.3-4}
Let $S$ be a semigroup with apartness. The following conditions are
true:
\begin{itemize}
\item[\emph{(i)}] Let  $T$ be an sd-convex subset of a semigroup with apartness
$S$. If ${\sim}T$ is inhabited, then it is  an ideal of $S$.
\item[\emph{(ii)}] If $I$ is an sd-completely isolated ideal of a semigroup with apartness $S$,
then ${\sim} I$ is a convex subsemigroup of $S$.
\end{itemize}
\end{lem}
\begin{proof}
(i). \ Let $x,y,\in {\sim}T$. Let $a\in {\sim}T$ and $x\in S$. By
the assumption we have that
 $ax\in T$ or $ax\in\sim T$. If $ax\in T$,
 then, as $T$ is convex, we have $a\in
T$, which is impossible. Similarly, one can prove that $xa\in
{\sim}T$. So, ${\sim}T$ is an ideal of $S$.

\smallskip
(ii).\ In a similar manner as in (i) we can prove that ${\sim} I$ is
a subsemigroup of $S$.

Let $xy\in {\sim} I$. By the assumption, we have $x\in I$ or
$x\in\sim I$. If $x\in I$, then, as $I$ is an ideal, we have $xy\in
I$, which is impossible. Thus $x\in\sim  I$. Similarly, we can prove
that $y\in {\sim} I$. So, ${\sim} I$ is convex.
\end{proof}

Let us start with an example of a co-quasiorder defined on a
semigroup with apartness $S$.

\begin{ex}\label{ex-bhm1}
Let $S$ be a semigroup given by
\begin{center}
\begin{tabular}{c|c c c c c c} \label{tab:1}
  $\cdot$ & a & b & c & d & e \\\hline
  a & b & b & d & d & d \\
  b & b & b & d & d & d \\
  c & d & d & c & d & c \\
  d & d & d & d & d & d \\
  e & d & d & c & d & c
\end{tabular}
\end{center}
Let the equality on $S$ be  the diagonal $\triangle_S = \{ (a,a),
(b,b), (c,c), (d,d) , (e,e)\}.$ If we denote  by
 $K=\triangle_S \cup \{(a,b),(b,a)\}$, then we can define an apartness $\#$ on
 $S$  by  $(S\times S) \setminus  K$.
The  relation $\tau \subseteq S \times S$, defined by
$$\tau  = \{(c,a),(c,b),(d,a),(d,b),(d,c),(e,a),(e,b),(e,c), (e,d)\},$$
 is a co-quasiorder  on $S$.
\end{ex}

Let $\tau$ be a co-quasiorder defined on a semigroup $S$ with
apartness.  Following the classical results as much as possible,  we
can start with the following definition.

\medskip
A co-quasiorder $\tau$  on a semigroup $S$  is
\begin{itemize}
\itemsep=0.95pt
  \item \emph{complement positive} if $(a,ab), (a,ba)\in \sim\tau $ for any $a,b\in S$,
  \item with \emph{constructive common multiple property}, or, in short,
  with \emph{constructive cm-property}
   if $(ab,c) \in \tau \Rightarrow(a,c)\in \tau
\vee (b,c)\in \tau$  for all $a,b,c\in S$,
  \item with \emph{complement common multiple property}, or, in short,
  with \emph{complement cm-property} if $(a,c), (b,c)\in \sim \tau \Rightarrow
  (ab,c)\in \sim\tau$  for all $a,b,c\in S$.
\end{itemize}
Recall, by the Proposition~\ref{senegcompl}, $(\sim\tau)=
(\neg\tau)$.

\begin{ex}\label{ex2}
The co-quasiorder $\alpha$ defined on the semigroup $S$ considered
in Example~\ref{ex-bhm1} is not complement positive because we have
$(e,ea)=(e,d)\in \alpha$.
\end{ex}

\begin{ex}\label{pcq-1}
Let $S$ be the three element semilattice given by
\begin{center}
\begin{tabular}{c|c c c  } \label{tab:2}
  $\cdot$ & a & b & c \\\hline
  a & a & c & c  \\
  b & c & b & c \\
  c & c & c & c
\end{tabular}
\end{center}
Let the equality on $S$ be  the diagonal $\triangle_S = \{ (a,a),
(b,b), (c,c)\}$.
 We can define an apartness $\#$ on $S$
 to be  $(S\times S) \setminus \triangle_S $.
 Thus, $(S, = , \#, \cdot)$ is a
 semigroup with apartness.  The
 relation $\tau \subseteq S \times S$, defined by
$$\tau  = \{(a,b), (c,a),(c,b)\},$$
 is a complement positive co-quasiorder  on $S$.

\medskip
On the  other  hand, from $(ab,a)=(c,a)\in\tau$ neither $(a,a)$ nor
$(b,a)$  are in $\tau $, so $\tau$ does not have the constructive
cm-property. From $(a,a)\bowtie\tau$ and $(b,a)\bowtie\tau$, we have
$(ab,a)=(c,a)\in\tau$, and $\tau$ does not have the complement
cm-property as well.
\end{ex}

The following lemma shows how some sd-subsets lead us to positive
co-quasiorders.

\begin{lem}\label{lem4.1}
Let $S$ be a semigroup with apartness $S$.
\begin{itemize}
\item[\emph{(i)}] If $K$ is an sd-convex subset of  $S$, then the relation
$\tau$ defined by
$$
(a,b)\in \tau \stackrel{\rm def}{\Leftrightarrow} a \in\sim K \wedge
b \in K
$$
is a complement positive co-quasiorder on $S$.
\item[\emph{(ii)}] If $J$ is an sd-ideal of $S$ such that $J \subset S$, then the relation
$\tau$  defined by
$$
(a,b)\in \tau \stackrel{\rm def}{\Leftrightarrow} a \in J \wedge  b
\in\sim J
$$
is a complement positive  co-quasiorder  on $S$.
\end{itemize}
\end{lem}
\begin{proof}
(i). \ By Lemma \ref{lem3.1}, $\tau$ is a co-quasiorder on $S$. Let
$(x,y)\in\tau$.  By the co-transitivity of $\tau$, we have $(x,a)\in
\tau \,\vee \,(a,ab)\in \tau \,\vee \,(ab,y)\in \tau$, for any
$a,b\in S$. If $(a,ab)\in \tau$, then, by the definition of $\tau$,
we have $a\in\sim K$ and $ab\in K$, and, as $K$ is a convex subset,
we have $a\in K$ and $b\in K$, which is impossible. So, we have $
(x,a)\in \tau \,\vee \,(ab,y)\in \tau$. By the strong irreflexivity
of $\tau$ we have $x \# a \,\vee \, ab\#y$,
 i.e. $(x,y)\# (a,ab)$. Thus, we have proved that
$(a,ab)\bowtie\tau$ for any $a,b\in S$. The proof of
$(a,ba)\bowtie\tau$ is similar. Therefore, $\tau$ is a complement
positive co-quasiorder on $S$.

\smallskip
(ii). \ By Lemma~\ref{lem3.1}, $\tau$ is a co-quasiorder on $S$. Let
$(x,y)\in\tau$. By the co-transitivity of $\tau$,  we have $(x,a)\in
\tau \,\vee \,(a,ab)\in \tau \,\vee \,(ab,y)\in \tau,$
 for any $a,b\in S$. If $(a,ab)\in \tau$, then, by the
definition of $\tau$, we have $a\in J$ and $ab\in\sim J$, which, as
$J$ is an ideal, further implies $ab\in J$, which is a
contradiction. The rest of the proof is similar to the arguments in
the proof of (i).
\end{proof}

By Proposition~\ref{p11-bhm}, if  any left/right-class of a
co-quasiorder defined on a set with apartness is a (strongly)
detachable subset, then \textbf{LPO} holds. This shows that Theorem
4.1 on a complement positive co-quasiorder (and Lemma 3.2 important
for its proof) from \cite{cmr2} cannot be proved outside
intuitionistic logic of constant domains \textbf{CD}. Nevertheless,
we can prove within intuitionistic logic the next theorem, which is
its weaker version, and another important result of this section.
The description of a complement positive co-quasiorder via its
classes follows.
\begin{teo}\label{thm4.11}
Let $\tau$ be  a co-quasiorder $\tau$ on a semigroup $S$.
\begin{itemize}
\item[\emph{(i)}] If $\tau$ is complement positive, then
$$\forall _{a,b \in S}\ \  (\tau(ab)\subseteq \tau a \cap \tau b).$$
\item[\emph{(ii)}] If $\tau a$ is an sd-ideal  of $S$ and $a \bowtie \tau a$ for every $a \in
S$, then $\tau$ is complement positive and
$$\forall _{a,b \in S}\ \ (a \tau \cup b \tau \subseteq (ab)\tau
).$$
\item[\emph{(iii)}] If $a\tau$ is an sd-convex subset of $S$, and  $a \bowtie a\tau$ for every $a \in
S$, then $\tau$ is a complement  positive co-quasiorder.
\end{itemize}
\end{teo}
\begin{proof}
(i).  \  Let $\tau$ be a complement positive co-quasiorder. For all
$a,b,x\in S$ such that $x \in \tau(ab)$, that is $ (x,ab)\in \tau$,
by the co-transitivity of $\tau$, we have
$$
((x,a)\in \tau \vee (a,ab)\in \tau )\wedge  ((x,b)\in \tau \vee
(b,ab)\in \tau).
$$
But, $\tau$  is complement positive, so that we have
 $(x,a)\in \tau \wedge
(x,b)\in \tau$, i.e. $x\in \tau a \cap \tau b$.

\smallskip
(ii). \  Let $(x,y)\in\tau$  and $a,b \in S$. Then, by the
co-transitivity of $\tau$,
$$
(x,a)\in \tau  \vee (a,ab)\in \tau  \vee (ab,y)\in \tau.
$$
If $a\in \tau (ab)$, then, as $\tau (ab)$ is an ideal, we have
$ab\in\tau (ab)$, which is, by the assumption, impossible. Now, by
the strong irreflexivity of $\tau$, we have $x\# a$ or $ab\# y$,
that is $(x,y)\# (a,ab)$. Thus, $(a,ab)\bowtie \tau$ for any $a,b\in
S$. $(a,ba)\bowtie\tau$ can be proved similarly. Thus, $\tau$ is a
complement  positive co-quasiorder.

\medskip
Let $x \in a \tau \cup b$, $x\in S$. By the co-transitivity and
complement positivity of $\tau$, we have
\begin{align*}
x \in a \tau \cup b \tau &\Leftrightarrow x \in a \tau \vee  x \in b \tau \\
&\Leftrightarrow (a,x)\in \tau \vee (b,x)\in \tau\\
&\Rightarrow ((a,ab)\in \tau \vee (ab,x)\in \tau) \vee ((b,ab)\in \tau \vee (ab,x)\in \tau )\\
&\Rightarrow (ab,x)\in \tau\\
& \Leftrightarrow x \in (ab)\tau.
\end{align*}

(iii). \ Let $(x,y)\in\tau$. Then, by the co-transitivity of $\tau$,
$$
(x,a)\in \tau \vee (a,ab)\in \tau  \vee (ab,y)\in \tau,
$$
for any $a,b\in S$. Let $(a,ab)\in \tau$, that is $ab\in a\tau$.
Then, by  assumption, $a\in a\tau$ (and $b\in a\tau$), which is
impossible. Now, by the strong irreflexivity of $\tau$, we have $x\#
a$ or $ab\# y$, that is, $(x,y)\# (a,ab)$. Thus $(a,ab)\bowtie \tau$
for any $a,b\in S$. $(a,ba)\bowtie \tau$ can be proved similarly.
Thus, $\tau$ is a complement positive co-quasiorder.
\end{proof}

\begin{teo} \label{cor4.1}
A complement positive co-quasiorder with the constructive
cm-property has the complement cm-property.
  \end{teo}
  \eject
\begin{proof}
Let $\tau$ be a complement positive co-quasiorder with the
constructive cm-property on a semigroup $S$ and let $a,b,c,x,y\in S$
be such that $(a,c),(b,c)\bowtie \tau$ and  $(x,y)\in\tau$. Then we
have
\begin{align*}
(x,y)\in \tau &\Rightarrow (x,ab)\in \tau \vee (ab,c)\in \tau \vee (c,y)\in \tau&&\text{by co-transitivity}\\
&\Rightarrow x \#ab \vee (a,c)\in \tau \vee (b,c)\in \tau \vee c \#
y&&\text{by strong reflexivity}\\
&   &&\text{ and by constructive cm-property}\\
&\Rightarrow (ab,c) \# (x,y)&&\text{since $(a,c)\bowtie \tau$ and
$(b,c)\bowtie \tau$.}
\end{align*}
Hence $(ab,c)\bowtie \tau$, i.e. $(ab,c)\in\sim\tau$.
\end{proof}

\subsubsection{Intuitionistic logic of constant domains CD as a background}

If we have a complement positive co-quasiorder $\tau$ on a semigroup
with apartness $S$, we can construct special subsets and semigroups
mentioned above. Some other criteria for a co-quasiorder to be
complement  positive will be given too.

\begin{teo}\label{thm4.1}
The following conditions for a co-quasiorder $\tau$ on a semigroup
$S$ are equivalent:
\begin{itemize}
\itemsep=0.9pt
\item[\emph{(i)}]  $\tau$ is complement positive;
\item[\emph{(ii)}]  $\forall _{a,b \in S}\ \  (a \tau \cup b \tau \subseteq (ab)\tau )$;
\item[\emph{(iii)}] $\forall _{a,b \in S}\ \  (\tau(ab)\subseteq \tau a \cap \tau b)$;
\item[\emph{(iv)}] $a\tau$ is an sd-convex subset of $S$ and  $a \bowtie a\tau$ for every $a \in S$;
\item[\emph{(v)}] $\tau a$ is an sd-ideal  of $S$ and $a \bowtie \tau a$ for every $a \in S$.
\end{itemize}
\end{teo}
\begin{proof}
(i) $\Rightarrow$ (iii), (v) $\Rightarrow$ (i), (v) $\Rightarrow$
(ii), (iv) $\Rightarrow$ (ii).  Those implications are proved in the
Theorem~\ref{thm4.11}.

\smallskip
(iii) $\Rightarrow$ (iv). By Lemma~\ref{lem3.2}, $a\tau$ is an
sd-subset of $S$ such that $a\bowtie a\tau$ for any $a\in S$. We
have
\begin{align*}
xy \in a \tau &\Leftrightarrow (a,xy)\in \tau\\
              &\Leftrightarrow a\in  \tau (xy) \subseteq \tau x \cap \tau y\\
              &\Rightarrow a \in \tau x \wedge  a \in \tau y\\
              &\Leftrightarrow x \in a \tau \wedge y \in a \tau.
\end{align*}
So, $a\tau$ is an sd-convex subset for any $a\in S$.

\smallskip
(i) $\Rightarrow$ (v). \   By Lemma~\ref{lem3.2}, $\tau a$ is an
sd-subset of $S$ such that $a\bowtie \tau a$ for any $a\in S$. Let
$a,x\in S$ be such that $x\in \tau a$, i.e. $(x,a)\in \tau$. By the
co-transitivity of $\tau$, we have
$$
(x,xs)\in \tau \vee (xs,a)\in \tau,
$$
for any $s\in S$. But, as $\tau$ is positive, we have only
$(xs,a)\in\tau$, i.e $xs\in\tau a$. In the same way one can prove
that $sx\in \tau a$. Thus, $\tau a$ is an ideal of $S$ for any $a\in
S$.

\smallskip
(ii) $\Rightarrow$ (v). \  By Lemma~\ref{lem3.2}, $\tau a$ is an
sd-subset of $S$, and $a\bowtie \tau a$ for any $a\in S$. Now, let
$x\in \tau a$ and $s\in S$. Then, by the co-transitivity of $\tau$,
we have $(x,xs)\in \tau$ or $(xs,a)\in\tau$. If $(x,xs)\in \tau$,
then $xs\in x\tau \subseteq x\tau \cup s\tau \subseteq  (xs)\tau$,
which is, by Lemma~\ref{lem3.2}, impossible. Thus $(xs,a)\in\tau$.
As $(sx,a)\in\tau$ can be proved similarly, we have proved that
$\tau a$ is an sd-ideal of $S$.
\end{proof}

\begin{teo}  \label{thm4.2}
Let $\tau$  be a complement positive co-quasiorder on a semigroup
$S$. The following conditions are equivalent:
\begin{itemize}
\item[\emph{(i)}] $\tau$ has the constructive cm-property;
\item[\emph{(ii)}]  $\forall _{a,b\in S}\ ((ab)\tau = a \tau  \cup  b \tau )$;
\item[\emph{(iii)}] $\tau a$ is an sd-completely isolated ideal of $S$ such that $a\bowtie\tau
a$ for any $a\in S$.
\end{itemize}
\end{teo}
\begin{proof}
(i) $\Rightarrow$ (ii). \ By Theorem~\ref{thm4.1}, $a\tau\cup b\tau
\subseteq (ab)\tau$ for all $a,b\in S$. To prove the converse
inclusion, take $x\in (ab)\tau$. Then we have
\begin{align*}
 x \in (ab)\tau &\Leftrightarrow (ab,x)\in \tau&&\\
&\Rightarrow (a,x)\in \tau \vee (b,x)\in \tau&&\text{by the cm-property}\\
&\Leftrightarrow x \in a \tau \vee  x \in b \tau\\
&\Leftrightarrow x \in a\tau  \cup  b \tau.
\end{align*}

(ii) $\Rightarrow$ (iii). \  By Theorem~\ref{thm4.1}, $\tau a$ is an
sd-ideal of $S$ such that  $ a\bowtie \tau a$, for any $a\in S$. Let
$x,y\in S$  be such that $xy \in \tau a$. Then $a \in(xy)\tau = x
\tau \cup  y \tau$ by the assumption. Thus, $a \in x \tau$ or $a \in
y \tau$. So, $x \in \tau a$ or $y \in \tau a$, and $\tau a$ is an
sd-completely isolated ideal of $S$ for any $a\in S$.

\smallskip
(iii) $\Rightarrow$ (i). Let $a,b,c\in S$ be such that $(ab,c)\in
\tau $. Then, $ab \in  \tau c$ and, since $\tau c$ is completely
isolated, $a \in \tau c$ or $b\in  \tau c$, which means that
$(a,c)\in \tau$  or $(b,c)\in \tau$.
\end{proof}

Following Bishop, every classical theorem presents the challenge:
find a constructive version with a constructive proof.  This
constructive version can be obtained by strengthening the conditions
or weakening the conclusion of the theorem. There are, often,
several constructively different versions of the same classical
theorem.

\medskip
Comparing the obtained results for complement positive
co-quasiorders with the parallel ones for positive quasiorders in
the classical background, we can conclude that the classical
Theorem~\ref{bsc1} breaks into two new ones in the constructive
setting:
\begin{itemize}
\itemsep=0.9pt
\item Theorem~\ref{thm4.11} obtained by  weakening the
conclusions,
\item Theorem~\ref{thm4.1}   obtained by
strengthening the conditions - here strengthening the logical
background.  Recall that intermediate logic proves
intuitionistically as well as   classically valid theorems, yet they
often possess a strong constructive flavour.
\end{itemize}
In addition, there are two definitions: those of the constructive
cm-property, and the complement  cm-property.  Nevertheless, the
last definition, Theorem~\ref{cor4.1}, is stronger.

\begin{rem}
For some classical theorems it is shown that they are not provable
constructively. Some  classical theorems are neither provable nor
disprovable, that is, they are independent of \rm{\textbf{BISH}}.
\end{rem}

\subsubsection{QSP for semigroups with apartness}\label{s-322}

Let us remember that in \textbf{CLASS} the compatibility property is
an important condition for providing the semigroup structure on
quotient sets. Now we are looking for the tools for introducing an
apartness relation on a factor semigroup. Our starting point is the
results from   Subsection~\ref{s-313}, as well as the next
definition.

\medskip
A co-equivalence $\pmb{\kappa}$ is a \emph{co-congruence} if it is
\emph{co-compatible}
$$\forall _{ a,b,x, y\in S} \, ((ax,by)\in\kappa \ \Rightarrow \
(a,b)\in\kappa \vee (x,y)\in\kappa )$$

\begin{teo}  \label{cmrth3}
Let $S$ be a semigroup with apartness. Then
\begin{itemize}
\itemsep=0.95pt
\item[\emph{(i)}] Let $\mu$ be a congruence, and $\kappa$ a co-congruence on $S$. Then, $\kappa$
  defines an apartness  on the factor set $S/\mu$ if and only if $\mu\cap\kappa =
  \emptyset$.
\item[\emph{(ii)}]  The quotient mapping $\pi : S \rightarrow
S/\mu$, defined by $\pi (x)  = x\mu$, is an onto se-homomorphism.
\end{itemize}
\end{teo}
\begin{proof}
(i). \ If $\kappa$ defines an apartness on $S/\mu$, then, by Theorem~\ref{BHM-1}(i), $\mu\cap\kappa=\emptyset$.

Let $\mu$ be a congruence and $\kappa$ a co-congruence on a
semigroup with apartness $S$ such that $\mu\cap\kappa=\emptyset$.
Then, by Theorem~\ref{BHM-1}(i), $\kappa$ defines apartness
on$S/\mu$.

Let $a\mu\, x\mu \# b\mu \, y\mu$, then $(ax)\mu \# (bx)\mu$ which further, by the definition of apartness on $S/\mu$, ensures that
$(ax,by)\in\kappa$. But $\kappa$ is a co-congruence, so either $(a,b)\in\kappa$ or $(x,y)\in\kappa$. Thus, by the definition of
apartness in $S/\mu$ again, either $a\mu \# b\mu$ or $x\mu \# y\mu$. So $(S/\mu ,=,\#,\cdot\,)$ is a semigroup with apartness.

\smallskip
(ii).\ By Theorem~\ref{BHM-1}(ii),  $\pi$ is an onto se-mapping. By
(i) and   assumption, we have
\[
\pi(xy)=(xy)(\sim\kappa)=x(\sim\kappa)\,y(\sim\kappa)=\pi(x)\pi(y).
\]
Hence $\pi$ is a homomorphism.
\end{proof}

As a consequence of Theorem~\ref{cmrth3} and Corollary~\ref{mhb-c32}
we have the next corollary.

\begin{cor}\label{mhb-c1}
If $\kappa$ is a co-congruence on $S$, then the relation
${\sim}\kappa(={\neg} \kappa)$ is a congruence on $S$, and $\kappa$
defines an apartness on $S/\sim\kappa$.
\end{cor}

\emph{The apartness isomorphism theorem for semigroups with
apartness} follows.

\begin{teo} \label{cmrth4}
Let $f: S\rightarrow T$ be an se-homomorphism between semigroups
with apartness. Then:
\begin{itemize}
\item[\emph{(i)}]   $\mathrm{coker}\,f$ is a co-congruence on
$S$, which defines an apartness on  $S/\ker\,f$,
\item[\emph{(ii)}]  the
mapping $\theta:S/\ker \,f\to T$, defined by
$\theta(x(\ker\,f))=f(x)$, is an apartness embedding such that
$f=\theta\circ\pi$; and
\item[\emph{(iii)}] if $f$ maps $S$ onto $T$, then
$\theta$ is an apartness isomorphism.
\end{itemize}
\end{teo}

\begin{proof}
See \cite{mmscr}.
\end{proof}

Recall, following \cite{fr1}, \textbf{BISH} (and constructive
mathematics in general) is not the study of constructive things, it
is a constructive study of things. In constructive proofs of
classical theorems, only constructive methods are used.

Although  constructive theorems might look like the corresponding
classical versions, they often have more complicated hypotheses and
proofs. Comparing Theorem~\ref{sth2} (respectively
Theorem~\ref{sth21}) for classical semigroups and
Theorem~\ref{cmrth4} (respectively Theorem~\ref{cmrth3}) for
semigroup with apartness, we have  evidence for that.

\section{Concluding remarks}\label{s-4}

During the implementation of the FTA Project \cite{gpwz2}, the
notion of commutative constructive semigroups with tight apartness
appeared. We put noncommutative constructive semigroups with
``ordinary'' apartness in the centre of our study, proving first, of
course, that such  semigroups do exist. Once again we want to
emphasize that semigroups with apartness are a \textbf{new
approach}, and not a new class of semigroups.

\medskip
Let us give some examples of applications of  ideas presented in the
previous section. We will start with constructive analysis. The
proof of one of the directions of the constructive version of the
Spectral Mapping Theorem is based on some elementary constructive
semigroups with inequality techniques, \cite{bh}. It is also worth
mentioning the applications of commutative basic algebraic
structures with tight apartness within the automated reasoning area,
\cite{gc}. For possible applications within computational linguistic
see \cite{mam}.  Some topics from mathematical economics can be
approached constructively too (using some order theory for sets with
apartness), \cite{bb}. Contrary to the classical case, the
applications of constructive semigroups with apartness, due to their
novelty, constitute an unexplored area. In what follows some
possible connections between semigroups with apartness and computer
science are sketched, \cite{mmscr}.

\medskip\smallskip

semigroups with apartness           \quad\quad\quad\quad semigroups
with apartness

$\updownarrow$
\quad\quad\quad\quad\quad\quad\quad\quad\quad\quad\quad\quad\quad\quad\quad\quad
$\updownarrow$

bisimulation \quad\quad\quad\quad\quad\quad\quad\quad\quad\quad
automated theorem proving

$\updownarrow$
\quad\quad\quad\quad\quad\quad\quad\quad\quad\quad\quad\quad\quad\quad\quad\quad
$\updownarrow$

formal reasoning about processes    \quad knowledge representation
                                            and automated reasoning

$\updownarrow$
\quad\quad\quad\quad\quad\quad\quad\quad\quad\quad\quad\quad\quad\quad\quad\quad
$\updownarrow$

process algebra \quad\quad\quad\quad\quad\quad\quad\quad\quad
artificial intelligence

$\updownarrow$

transactions and concurrency

$\updownarrow$

databases

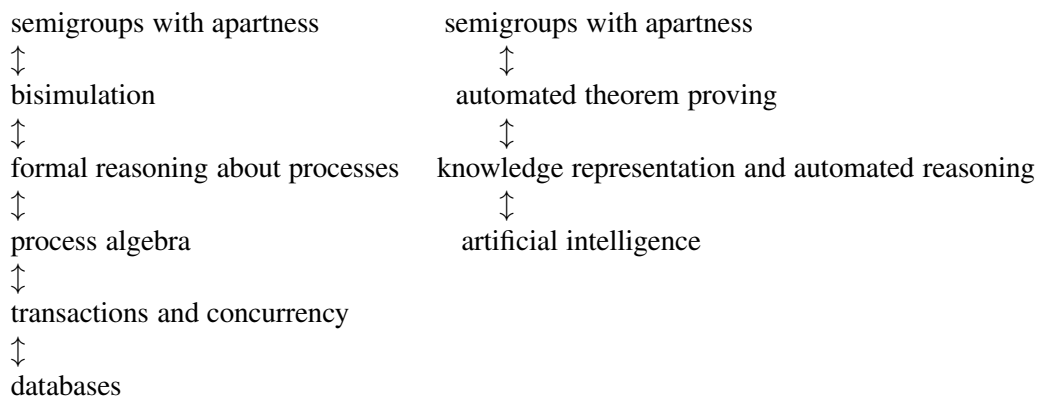
\smallskip

\vfil\eject

\noindent One of the directions of future work is to be able to say
more about those links. The study of basic constructive algebraic
structures with apartness as well as constructive algebra as a whole
can impact the development of other areas of constructive
mathematics. On the other hand, it can make both proof engineering
and programming more flexibile.

Although the classical theory of semigroups  has been considerably
developed in the last decades, constructive mathematics has not paid
much attention to   semigroup theory. One of our main scientific
activities will be to further develop  of the constructive theory of
semigroups with apartness. Semigroups will be examined
constructively,  that is with intuitionistic logic. To develop this
constructive theory of semigroups with apartness, we need first to
clarify the notion of a set with apartness. The initial step towards
grounding the theory is done by our contributing papers \cite{cmr1},
\cite{cmr2}, \cite{mmss2}, \cite{mmscr}, \cite{md}  - a critical
review of some of  those results as well as the solutions to some of
the open problems arising from those papers are presented in
Section~\ref{s-3}.

Why should a mathematician choose to work in this manner? As it is
written in one of the reviews of  Errett Bishop's monograph
\emph{Foundations of functional analysis}, \cite{gs}, ``to replace
the classical system by the constructive one does not in any way
mutilate the great classical theories of mathematics. Not at all. If
anything, it strengthens them, and shows them, in a truer light, to
be far grander than we had known.'' At heart, Bishop's constructive
mathematics is simply mathematics done with intuitionistic logic,
and may be regarded as ``constructive mathematics for the working
mathematician'', \cite{t1}. The main activity in the field consists
in proving theorems rather than demonstrating the unprovability of
theorems (or making other metamathematical observations),
\cite{mjb}. ``Theorems are tools that make new and productive
applications of mathematics possible,'' \cite{jtw}.

The theory of semigroups with apartness is, of course, in its
infancy, but, as we have already pointed out, it promises  a
prospective of applications in other (constructive) mathematics
disciplines, certain areas of computer science, social sciences,
economics.

To conclude, although one of the main motivators for initiating and
developing the theory of semigroups with apartness comes from the
computer science area, in order to have profound  applications, a
certain  amount of the theory, which  can be applied, is necessary
first. Among priorities, besides the growing the general theory, are
further developments of: constructive relational structures -
(co)quotient structures in the first place, constructive order
theory, theory of ordered semigroups with apartness, etc.

The  Summary of the European Commission's  \emph{Mathematics for
Europe}, June 2016, \cite{ec}, states that ``mathematics should not
only focus on nowadays' applications but should leave room for
development, even theoretical, that may be vital tomorrow.'' With a
strong  belief in the tomorrow's vitalness of the theory of
semigroups with apartness, the focus should be  on its further
development. On the other hand, it is useful to ``leave room'' for
``nowadays' applications'' as well. All those will  represent the
core of our forthcoming papers.

\subsection*{Acknowledgements}

   The authors are grateful to anonymous referees for careful reading   of the manuscript and helpful comments. M. M. is supported
   by the Faculty of Mechanical Engineering, University of Ni\v s, Serbia, Grant ``Research and development of new generation machine
   systems in the function of the technological development of Serbia''. M. N. H. is supported by TWAS Research
   Grant RGA No. 17-542 RG / MATHS / AF / AC \_G -FR3240300147. The ICMPA-UNESCO Chair is in partnership
   with Daniel Iagolnitzer Foundation (DIF), France, and the Association pour la Promotion Scientifique de l'Afrique (APSA),
   supporting the development of mathematical physics in Africa.



\end{document}